\documentclass[12pt,reqno]{amsart}
\usepackage[T1]{fontenc}
\usepackage{amsmath}
\usepackage{amssymb}
\usepackage{amsthm}
\usepackage{enumitem}
\usepackage{hyperref}
\usepackage{tikz-cd}
\usetikzlibrary{patterns}
\usepackage{setspace}

\usepackage{comment}

\usepackage{soul}

\newtheorem*{theorem*}{Theorem}

\if@dg@section
\def\sectionautorefname~{\S}

\fi

\newtheorem{theorem}{Theorem}[section]
\newtheorem{definitio}[theorem]{Definition}
\newenvironment{definition}{\begin{definitio} \rm }{\end{definitio}}
\newtheorem{rem}[theorem]{Remark}

\newenvironment{remark}{\begin{rem} \rm }{\end{rem}}
\newtheorem{ex}[theorem]{Example}
\newenvironment{example}{\begin{ex} \rm }{\end{ex}}
\newtheorem{lemma}[theorem]{Lemma}
\newtheorem{proposition}[theorem]{Proposition}
\newtheorem{corollary}[theorem]{Corollary}

\newcommand{\F}{\mathbb{F}}

\newcommand{\Z}{\mathbb{Z}}
\newcommand{\Q}{\mathbb{Q}}
\newcommand{\R}{\mathbb{R}}
\newcommand{\C}{\mathbb{C}}

\newcommand{\rank}{\text{rank}}

\newcommand{\diag}{\operatorname{diag}}

\newcommand{\ind}{\operatorname{ind}}
\newcommand{\alg}{\operatorname{alg}}
\newcommand{\bsat}{\beta_{\operatorname{sat}}}
\newcommand{\Arf}{\operatorname{Arf}}
\newcommand{\im}{\operatorname{im}}
\newcommand{\id}{\operatorname{id}}
\newcommand{\lk}{\operatorname{lk}}

\newcommand{\hasse}{\varepsilon}

\def\transpose#1{{#1}^\top}
\newcommand{\std}{\operatorname{Std}}
\newcommand{\eFM}{\overset{\text{\sc fm}}\sim}
\newcommand{\Disc}{\operatorname{Disc}}

\definecolor{vert}{RGB}{0,205,0}

\title{Seifert forms and slice Euler characteristic of links}

\author[V. Florens]{Vincent Florens}
\address{Université de Pau, Pau, France (V.\,F.)}

\author[S.\,Yu. Orevkov]{Stepan Orevkov}
\address{IMT, l'universit\'e de Toulouse, Toulouse, France (S.\,O.)}
\address{Steklov Mathematical Institute, Moscow, Russia (S.\,O.)}

\date{}

\begin{document}

\begin{abstract}
We define the Witt coindex $\beta$ of a rational bilinear form (not necessarily symmetric)
as the codimension minus dimension of a maximal totally isotropic subspace.
Levine gave a computable necessary and sufficient condition for $\beta=0$, but
no algorithm of computation of $\beta$ had been known when $\beta>0$.
Our main result is a computable necessary and sufficient condition for $\beta=1$
in the case when the symmetrized form is non-degenerate.
As shown by Taylor, all upper bounds of the slice genera of knots based on
their Seifert form can be derived from the Witt coindex of the Seifert form.
The case of an odd $\beta$ does not apply to knots but we show that it provides an estimate
of the slice Euler characteristic for the links with an even number of components.
We illustrate our results with examples where the Witt coindex is used to show that
a two-component link does not bound a locally flat cylinder in the four-ball,
whereas any other known restriction does not show it.
\end{abstract}

\maketitle

\section{Introduction}
A knot in $S^3$ is (topologically) \emph{slice} if it bounds a locally flat disk in $B^4$.
If $K$ is a slice knot then any Seifert form of $K$ is \emph{metabolic}
(it vanishes on a half-dimensional subspace). In this case $K$ is said to be \emph{algebraically slice}.
In \cite{Lev1, Lev2} Levine gave necessary and sufficient conditions
for a Seifert form of a knot to be metabolic.
These are:

\begin{enumerate}
\item the Alexander polynomial has the form $f(t)f(t^{-1})$
      (the Fox-Milnor condition);
\item the Tristram-Levine signature function vanishes;
\item certain invariants depending on a prime number $p$ vanish for all $p$.
\end{enumerate}

Livingston \cite{Liv} has shown that it is enough to check (3) for a finite number
of primes only. It follows that there is an algorithm to check whether a given knot
is algebraically slice or not.

The paper develops concordance for \emph{admissible} matrices over a field of characteristic zero (square matrices $A$ such that the Alexander polynomial $\Delta_A(t):=\det(A-t \transpose{A})$ does not vanish), generalizing Levine's framework. Given a square matrix $A$ of size $n$, we consider the Witt \emph{coindex} of $A$, defined as $$\beta(A)= n- 2 \operatorname{ind}(A).$$
Here ind$\, (A)$ denotes the maximal dimension of a subspace on which the associated bilinear form vanishes (a totally isotropic subspace). If $A$ is the Seifert matrix of a knot, then $\beta(A)$ coincides with the invariant $m$ introduced by Taylor \cite{Tay}.
 A generalized Witt cancellation lemma (see Lemma \ref{witt}) is used to show that  
$\beta$ is a concordance invariant, in fact even invariant under S-equivalence without the admissibility assumption (see Proposition \ref{S-equiv}).

 By definition $\beta(A)=0$ if and only if $A$ is metabolic, and the results of \cite{Lev1, Liv} provide an algorithm to decide whether
$\beta(A)=0$ where $A$ is a rational Seifert matrix of a knot.
Our main result is a continuation of Levine's work, addressing the next case $\beta(\theta)=1$.

\begin{theorem*}[Coindex-1 classification; see Theorem \ref{coindex1}]
Let $A$ be a matrix of odd dimension such that
the symmetrized matrix $A+\transpose{A}$ is non-singular (i.e. $\Delta_A(-1) \neq 0$).
Suppose $\beta(A)=1$. Then, the form
$A$ is concordant to an explicit standard matrix $\mathrm{Std}(\Delta_A)$, determined
by the coefficients of $\Delta_A$ (see Definition~\ref{def.standard}).
\end{theorem*}

In fact, Theorem \ref{coindex1} is formulated in a slightly more general form, which sometimes permits to use
a smaller matrix instead of $\mathrm{Std}(\Delta_A)$.
In order to prove these results, we consider \emph{isometric structures} $(V,b,T)$ where $b$ is a nonsingular
symmetric bilinear form on a finite dimensional space $V$ and $T$ is an isometry with respect to $b$.
Following Levine \cite[\S\S 8--9]{Lev1}, the Witt group of concordance classes of isometric structures $(V,b,T)$ is isomorphic to
the concordance group of forms with non-singular symmetrization, via
$A \mapsto (A+\transpose{A},\ A^{-1} \transpose{A})$. See Proposition \ref{prop.isom}.
This is the main tool in the proof of Theorem~\ref{coindex1}.

Combined with the results of Levine and Livingston cited above, Theorem~\ref{coindex1}
yields an algorithm to decide whether $\beta(A)=1$ for a given rational matrix $A$
of odd dimension  with $\Delta_A(-1)\ne 0$.
Briefly, the algorithm goes as follows.
\begin{enumerate}
\item
      Set $M_0 = A \oplus (-\mathrm{Std}(\Delta_A))$.
\item
      Proceeding as in \cite[\S8]{Lev1}, transform $M_0$ to an S-equivalent non-singular matrix $M$.
      Let $(Q,T)=(M+\transpose M,\ M^{-1}\transpose M)$
      be the associated isometric structure.
\item
      Decompose $(T,Q)=(T_{t-1},Q_{t-1})\oplus(T',Q')$ according to its $(t-1)$-primary component and the rest.
      Check if $Q_{t-1}$ is metabolic (if not, conclude that $\beta(A)\ne 1$).
\item
      Check if the isometric structure $(T',Q')$ is metabolic using the algorithm naturally derived from
      \cite{Lev1} and \cite{Liv}.
\end{enumerate}

This algorithm is described in more detail in Section \ref{sect.algo} and applied to several
examples in Section \ref{sect.examples}.

Theorem~\ref{coindex1} does not apply to knots for parity reasons, but it applies to
links with an even number of components.
We define the coindex $\beta(L)$ of an oriented link $L$ in
$S^3$ as the Witt coindex
of any Seifert form of $L$ on the homology with rational coefficients.
We show:

\begin{theorem*}[Concordance invariance; see Theorem \ref{main1}]
Let $L$ and $L'$ be oriented links in $S^3$ with non-zero Alexander polynomial.
If $L$ and $L'$ are (topologically, locally flat) concordant, then $\beta(L)=\beta(L')$.
\end{theorem*}

If $K$ is a knot, $\frac12\beta(K)$ coincides with
an invariant introduced by Taylor \cite{Tay}
and it is a lower bound for the slice genus of $K$ (also called the $4$-genus);
see \cite{Lev2,Tay}. In \S\ref{linkinv} we adapt the arguments from \cite{Lev2,Tay} to prove
the following result.

\begin{theorem*}[Bound on the slice Euler characteristic; see Theorem \ref{main2}]
For any oriented link $L$ in $S^3$,
$$\chi_s(L) \le 1-\beta(L),$$
where $\chi_s(L)$ denotes the slice Euler characteristic of $L$, i.e., the maximal Euler
characteristic of a locally flat, closed-component-free oriented surface in $B^4$ bounded by $L$.
\end{theorem*}

In particular, Theorem~\ref{coindex1} combined with Theorem~\ref{main2} yields the following fact:
if $\beta(L)\ne 1$ for a two-component link $L$, then $\chi_s(L) \le 1-\beta(L) < 0$, so $L$ cannot bound a
locally flat cylinder in the $4$-ball $B^4$.
Note $\chi_s(K)$ determines $g_s(K)$ for a knot $K$, however, none of these invariants determines the other for
links with more than one component (see the discussion after Theorem~\ref{main2}).

Let us mention that Taylor \cite{Tay} observed
that a Seifert form is realizable by a knot of slice genus $g$ if and only if
its coindex is less or equal to $2g$. This naturally motivates the definition 
of the {\it algebraic slice genus}
of a knot $K$ as $g_s^{\alg} = \frac12\beta(K)$.
In particular, $K$ is called {\it algebraically slice} if $g_s^{\alg}(K)=0$.
In view of Theorem~\ref{main2}, one might think that $1 - \beta(L)$ is a natural
definition of the algebraic slice Euler characteristic. 
However, the condition that the coindex of a Seifert form of a link is $1-\chi$
does not imply that the form is realizable by a link $L$ with $\chi_s(L)=\chi$.
For this reason, we do not know for the moment what should be the appropriate
notion of algebraic slice Euler characteristic. See \S\ref{sect.slice.chi}.

In \S\ref{wittgroups}, we define the concordance of admissible matrices and
give basic properties of the coindex. In \S\ref{isom} we introduce isometric structures.
\S\ref{sect.coind=1} is devoted to the main result concerning the forms of coindex 1.
In \S\ref{linkinv} we relate the topological and algebraic invariants discussed in the
paper. We conclude with \S\ref{sect.examples} where we illustrate our results by several
examples and discuss some related topics.


\section{Algebraic concordance and coindex} \label{wittgroups}

Let $\F$ be a field of characteristic zero (in fact, we need only the case when it is
the field of rational, real, or $p$-adic rational numbers: $\Q$, $\R$, or $\Q_p$).
In this section, we consider square matrices with entries in $\F$. 


\subsection{Normalized Alexander polynomial}\label{sect.alex}

Given a square matrix $A$, we define its normalized Alexander polynomial
(which we shall call just {\it Alexander polynomial\/}) as
$$
    \Delta_A(t) = t^{-k} \det(A - t\transpose{A}),
    \qquad\text{$k$ is such that $\Delta_A\in\F[t]$ and $\Delta_A(0)\ne 0$}
$$
(in particular, $\Delta_A(t) = \det(A - t\transpose{A})$ when $A$ is non-singular).
It is clear that
\begin{equation}\label{sym.Delta}
     \Delta_A(t) = (-t)^n \Delta_A(t^{-1}), \qquad n=\deg\Delta_A.
\end{equation}

One readily checks that $\Delta_A(t)$ is invariant under the S-equivalence, and then
the Alexander polynomial $\Delta_L(t)$ of a link $L$ is a well-defined invariant
(not up to multiplication by $c t^n$ for $c \in \F^*$). It can be equivalently defined as
$\Delta_L(t) = t^{n/2\,}\nabla_L(t^{1/2} - t^{-1/2})$ with $n=\operatorname{span}\nabla_L$, 
where $\nabla_L(z)$ is the Conway polynomial defined by
$\nabla_{L_+}(z)-\nabla_{L_-}(z)=z\nabla_{L_0}(z)$ and $\nabla_O(z)=1$.

Our choice of the sign of $\nabla_L$ corresponds to that in
\cite{Kau} and \cite[\S8.2]{KawBook} but it is opposite to the choice in
\cite[\S5.4]{KawBook} and \cite{Lic}
in the case of even-component links.


 \subsection{The group $\bar{\mathcal{G}}^\F$}\label{sect.witt.index}
 
\begin{definition}
A square matrix $A$  is \emph{admissible} if its Alexander polynomial is non-zero.
\end{definition}

In particular, any nonsingular matrix is admissible.
Two matrices $A$ and $A'$ are \emph{congruent} if there exists
a matrix $P$ with $\det P \neq 0$ such that $\transpose{P}\! A P= A'$.

With any square matrix $A$ we associate (in the standard way) a bilinear
form on an $\F$-vector space $V$ whose dimension equals the size of $A$.

\begin{definition}
A square matrix $A$ of size $2k$
is \emph{metabolic} if there exists a $k$-dimensional subspace $U$ of $\F^{2k}$,
totally isotropic for $A$, i.e., $\transpose{x}\! A y=0$ for all $x,y \in U$. 
Equivalently $A$ is congruent to a matrix of the form
$$
  \begin{pmatrix} 0 & B \\
                  C & D
  \end{pmatrix},
$$
where $B,C$ and $D$ are $k\times k$ matrices and $0$ is the zero matrix.
Two matrices $A_1$ and $A_2$ are \emph{concordant} if the orthogonal sum
$A_1 \oplus -A_2$ is metabolic. 
\end{definition}

The \emph{Witt index} of  $A$ is defined as the maximal dimension of a totally
isotropic subspace for $A$.%
\footnote{This is already a standard term for symmetric non-singular forms.
However, it is used very rarely for degenerate or non-symmetric forms,
and the definitions vary in these cases.}
We denote this index by $\ind (A)$ or $\ind_\F(A)$.
Note that if $\det A \neq 0$ and $A$ has size $n$,
then $\dim U \leq n/2$ for any totally isotropic subspace $U$.

\begin{proposition} \label{equiv}
The concordance of admissible matrices is an equivalence relation.
Congruent matrices are concordant.
\end{proposition}

\begin{proof}
The symmetry is obvious.
We have
$\transpose{P}\!\left(\smallmatrix\transpose{Q}\!\!AQ&0\\0&-A\endsmallmatrix\right)P
 =\left(\smallmatrix 0&*\\ *&*\endsmallmatrix\right)$ for
$P=\left(\smallmatrix I&I\\Q&0\endsmallmatrix\right)$, hence congruent matrices are concordant
(in particular, the concordance is reflexive).
The transitivity follows from Lemma~\ref{witt} below, which is a generalization of the
Witt cancellation lemma (cf.~\cite[\S3]{Lev2}).
\end{proof}
 
\begin{remark}
The concordance is no longer transitive if we omit the requirement that the
matrices are admissible. Indeed, according to our definition,
any matrix is concordant to the zero matrix of the same size.
It would be interesting to find
an adequate notion of concordance, relevant to Seifert matrices of
arbitrary links.
\end{remark}

Equivalence classes of admissible matrices form an abelian group under orthogonal sum,
denoted by $\bar{\mathcal{G}}^\F$, with subgroup $\mathcal{G}^\F$ the set of concordance
classes of matrices $A$ of even size with $\Delta_A(1)\Delta_A(-1)\ne 0$ introduced by Levine.
The metabolic class constitutes the identity element, and the inverse of the class of $A$ is the class of $-A$.

\begin{lemma} \label{witt}
Let $A$ and $B$ be square matrices.
Suppose that $B$ is admissible and metabolic of size $2k$.
Then
$$
    \ind \big( A \oplus B \big)= \ind (A) + k.
$$
In particular, if both $B$ and $A \oplus B$ are metabolic, then $A$ is metabolic.
\end{lemma}

\begin{proof}
In essential, the proof follows that of \cite[Lemma~1]{Lev2}
or the lemma on p.~87 in \cite{Ker}.
We first show that $\ind ( A \oplus B ) \geq \ind A + k$. Let $X$ and $Y$
be respectively the maximal isotropic subspaces for $A$ and $B$.
Denote $\theta$ and $\eta$ the forms on vector spaces $V$ and $W$ associated
with $A$ and $B$ respectively.
Then, for all $x,x' \in X$ and $y,y' \in Y$, one has
$$
   (\theta\oplus\eta)(x\oplus y,x'\oplus y')=\theta(x,x') + \eta(y,y') = 0.
$$
Hence $X \oplus Y$ is isotropic for $\theta \oplus \eta$ and 
$$
   \ind ( A \oplus B) \geq \dim ( X \oplus Y) = \dim X + \dim Y = \ind(A) + k.
$$

Let $Z$ be a supplementary subspace of $Y$ in $W$.
Let $U\subset V\oplus W$ be a totally isotropic subspace for $A \oplus B$,
with basis $(\alpha_1,\dots,\alpha_n)$. Suppose that $U$ is maximal, i.e.
$$
          n = \ind( A \oplus B).
$$
For all $i\in \{1,\dots,n \}$, there exist $x_i \in V$, 
$y_i \in Y$ and $z_i \in Z$ such that $\alpha_i=x_i+y_i+z_i$.
Let $r$ be the rank of the family $\{ z_i \}_i$. After re-ordering,
we can suppose that $z_1,\dots,z_r$ are linearly independent.
Then after the replacement of $\alpha_i$ with
$\alpha_i+\sum_{j=1}^r \lambda_{ij}\alpha_j$, $i={r+1},\dots,n$, for suitable
$\lambda_{ij}$, we may achieve that $z_{r+1}=\dots=z_n=0$.
Similarly, if $s$ is the rank of the family $\{x_i\}_{r < i \leq n }$,
we may achieve that $x_{r+1},\dots,x_{r+s}$ are linearly independent and
$x_{r+s+1}=\dots=x_n=0$. We obtain the following family of rank $n$:
\begin{align*}
& \alpha_i= x_i +  y_i+z_i, & i&=1,\dots,r, \\
& \alpha_i= x_i +  y_i,     & i&=r+1,\dots,r+s, \\
& \alpha_i =  \ \ \ \ \ \  y_i,     & i&=r+s+1,\dots,n. 
\end{align*}
For all $i,j \in \{r+1,\dots,r+s\}$, we have
$$
  \theta(x_i,x_j)=\theta(x_i,x_j)+\eta(y_i,y_j)=(\theta\oplus\eta)(\alpha_i,\alpha_j)=0.
$$
Hence there exists a subspace of $V$ of dimension $s$ on which $\theta$ vanishes.
It remains to show that
\begin{equation}\label{eq.transit}
                                  s \geq n-k.
\end{equation}

Since the polynomial $\Delta_B$ is non zero, there exists $\lambda \in \F$
such that $B^\lambda:=B + \lambda\transpose{B}$ is non-singular.
The family $\{y_i \}_{r+s < i \leq n}$ is linearly independent and the subspace $Y'$
of $Y$ spanned by this family has dimension $n-r-s$. Moreover, for all $i > r+s$ and
all $j$, we have
$$
    \eta(y_i,z_j)=\eta(y_i,y_j)+\eta(y_i,z_j)=\theta(0,x_j)+\eta(y_i,y_j+z_j)
               =(\theta\oplus\eta)(\alpha_i,\alpha_j)=0.
$$
and similarly $\eta(z_j,y_i)=0$. Hence, if $\eta^\lambda$ is the form associated with
$B^\lambda$ then
$$
  \eta^\lambda(y,z_j)=0 \text{ for } y \in Y' \text{ and } 1 \leq j \leq r,
$$
and $Y'$ is orthogonal to the subspace $\langle z_1,\dots,z_r \rangle$ of $Z$ with
respect to $\eta^\lambda$. If we extend the basis of $Y'$ to a basis
$(\tilde y_1,\dots,\tilde y_k)$ of $Y$, and $(z_1,\dots,z_r)$ to a basis of $Z$,
then $B^\lambda$ has an $(n-r-s)\times(k+r)$-block of zeros whose rows correspond to
$(y_{r+s+1},\dots,y_n)$ and columns correspond to
$(\tilde y_1,\dots,\tilde y_k,z_1,\dots,z_r)$. Since $B^\lambda$ is non-singular,
we conclude that $(n-r-s)+(k+r)\le 2k$, thus (\ref{eq.transit}) holds.
\end{proof}


\subsection{A concordance condition for proportional matrices}\label{sect.propor}

For a polynomial $F(t)\in\F[t]$, we denote $\bar F(t)=t^n F(t^{-1})$ where $n=\deg F$.

Let $F(t)\in\F[t]$, $F(0)\ne 0$. We say that $F$ is {\it symmetric} if $\bar F=cF$
with $c\in\F^\times$. Since $\bar{\bar F}=F$, we have
$c=\pm1$ in this case. Thus any symmetric polynomial $F$ is either {\it palindromic},
i.e. $F=\bar F$, or {\it antipalindromic}, i.e. $F=-\bar F$. 
Evaluating $F$ at $t=\pm 1$, we see that $F(t)$ is a multiple of $t-1$ when $F$
is antipalindromic, and a multiple of $t+1$ when $F$ is palindromic of odd degree.
The equation (\ref{sym.Delta}) means that the Alexander polynomial of an admissible
matrix of even (resp. odd) size is palindromic (resp. antipalindromic).

\begin{definition}
Two polynomials $P,Q\in\F[t]$ are {\it Fox-Milnor equivalent} if there exist
$F,G\in\F[t]$ with $F(0)G(0)\ne 0$, such that $F\bar FP=G\bar GQ$. We write in this case $P\eFM Q$.
\end{definition}

\begin{proposition} \label{meta}
If $A$ is a metabolic nonsingular matrix, then there exists $F \in \F[t]$ with 
$F(0) \neq 0$ such that $\Delta_A(t)= F(t) \bar F(t)$, in particular, $\Delta_A\eFM 1$.
\end{proposition}

\begin{proof}
Suppose that
$A = \begin{pmatrix} 0 & B \\
                     C & D \end{pmatrix}$, 
where the blocks have size $k \times k$. Then
$\Delta_A(t)=(-1)^k \det(\transpose{C} - t B)\det(B - t \transpose{C})
  = (-1)^k (-t)^k \det(B-t^{-1}\transpose{C}) \det(B-t \transpose{C})$.
Hence  $\Delta_A(t)=t^kF(t^{-1})F(t)= \bar{F}(t) F(t)$ (here $k=\deg F$
since $C$ is nonsingular). 
\end{proof}

\begin{lemma} \label{alexconc}
If two nonsingular matrices $A$ and $B$ are concordant, then $\Delta_A\eFM\Delta_B$.
\end{lemma}

\begin{proof}
Since $A \oplus -B$ is metabolic, by Proposition \ref{meta} there exists $F \in \F[t]$
with $F(0)\ne0$ such that $\Delta_A(t)\Delta_{-B}(t) =  F(t)\bar F(t)$. Let $k$ be the
dimension of $B$. Then $\Delta_{-B}=(-1)^k\Delta_B$ and (\ref{sym.Delta}) reads
$\Delta_B(t)=(-1)^k\bar\Delta_B(t)$. Thus
$$
   \Delta_A(t)= \frac{F(t)\bar F(t)}{(-1)^k\Delta_B(t)}=\Delta_B(t) \,
                 \frac{F(t)\bar F(t)}{(-1)^k\Delta_B^2(t)}= \Delta_B(t) \,
                  \frac{F(t)\bar F(t)}{\Delta_B(t) \bar{\Delta}_B(t)}\,.
$$
\end{proof}

\begin{lemma}\label{lem.propor} 
Let $A$ and $B$ be two nonsingular matrices of odd size such that 
$\Delta_B\eFM\mu\,\Delta_A$ with $\mu\in\F^\times$.
Suppose that  $\lambda A$ and $B$ are concordant for some $\lambda \in\F^\times$.
Then $\lambda=a^2\mu$ for some $a\in\F^\times$. In particular, $\mu A$ and $B$
are concordant.
\end{lemma}

\begin{proof}
Let $2n+1$ be the dimension of $A$. Then by Lemma \ref{alexconc}
$$
   \lambda^{2n+1}\Delta_A = \Delta_{\lambda A}\eFM\Delta_B\eFM\mu\Delta_A,
$$
hence $\lambda=a^2\mu$, $a\in\F^\times$. Note that $a^2\mu A$ is congruent
(and hence concordant) to $\mu A$.
\end{proof}


\subsection{The coindex of a concordance class}\label{sect.coind}
 
Remind that the Witt index $\ind_\F(A)$ of a matrix $A$ is the maximal dimension
of a totally isotropic subspace for $A$, see Section \ref{sect.witt.index}.

\begin{definition}
The \emph{Witt coindex} of a matrix $A$ of dimension $n$ is the integer
$$
     \beta(A)=\beta_\F(A)=   n - 2 \ind_\F (A).
$$
\end{definition}

By definition, $A$ is metabolic if and only if $\beta(A)=0$. Note also
that if $\det A \neq 0$, then $\beta(A)$ is a non-negative integer.
 
\begin{proposition}  \label{algconc}
If admissible matrices $A$ and $A'$ are concordant, then $\beta(A)=\beta(A')$.
\end{proposition}

\begin{proof}
Suppose that $A$ and $A'$ are concordant, which means that there exists
a metabolic $M$ with $A \oplus -A' \equiv M$. Hence  $A\oplus -A'\oplus A'\equiv M\oplus A'$.
By reflexivity, there exists $M'$ metabolic such that $ -A' \oplus A' \equiv M'$ and hence
$A\oplus M'\equiv M\oplus A'$. By Lemma \ref{witt}, it follows that $\beta(A)=\beta(A')$.
\end{proof}

\begin{example} \label{example.coind1}
Let $A=\left(\smallmatrix 0&B_{n+1}\\B_n&0\endsmallmatrix\right)^{-1}$ where
$B_n$ and $B_{n+1}$ are nonsingular square matrices of size
$n$ and $n+1$ respectively. Using the cofactor matrix we obtain $\beta(A)=1$.
\end{example}


\subsection{The coindex of an S-equivalence class}\label{sect.S-equiv}

We say that two square matrices are S-equivalent (over $\F$) if they are
obtained from each other by a sequence of congruences and the following replacements:
\begin{equation}\label{def.S-equiv}
    A \leftrightarrow
       B=\left(\begin{matrix}
                 A & 0 & u \\
                 0 & 0 & 0 \\
                 0 & 1 & 0 \end{matrix}\right),
    \qquad
    A \leftrightarrow
       C=\left(\begin{matrix}
                 A & 0 & 0 \\
                 0 & 0 & 1 \\
                 v & 0 & 0 \end{matrix}\right)
\end{equation}
where $u$ (resp. $v$) is a column (resp. row) with entries in $\F$.

\begin{proposition}\label{S=>conc} {\rm \cite[Claim in \S8]{Lev1} }
S-equivalent matrices are concordant.
\end{proposition}

This fact combined with Proposition~\ref{algconc} implies the invariance
of $\beta$ under S-equivalence of admissible matrices.
However, the following proposition shows that the invariance of $\beta$
under S-equivalence holds without the admissibility assumption and thus
it is  a well-defined link invariant (see \S\ref{coind(link)}).

\begin{proposition}  \label{S-equiv}
If two matrices $A$ and $A'$ are S-equivalent, then $\beta(A)=\beta(A')$.
\end{proposition}

\begin{proof}
The proof is almost the same as for Proposition~\ref{witt}. It is enough to show that $\ind B=\ind C=1+\ind A$ for $A$, $B$, and $C$
in (\ref{def.S-equiv}). We do it for $B$ (for $C$ the arguments are the same).
It is clear that $\ind B\ge 1+\ind A$. Let us show that $\ind B\le 1+\ind A$.

Let $\theta$ and $\Theta$ denote the forms on vector spaces $V$ and
$W=V\oplus Y\oplus Z$ associated with $A$ and $B$ respectively where
$Y$ and $Z$ are the one-dimensional subspaces of $W$ corresponding to
the last two rows and columns of $B$. Let $X$ and $U$ be the maximal
isotropic subspaces for $A$ and $B$. Let $(\alpha_1,\dots,\alpha_n)$ be a
basis of $U$, thus $n = \ind B$. Write $\alpha_i=x_i+y_i+z_i$ with $x_i\in V$,
$y_i\in Y$, $z_i\in Z$. We define $r$ and $s$ as in the proof of
Proposition~\ref{witt}, and we impose the same assumptions on
the vectors $y_i$'s and $z_i$'s (here $r\le 1$ and $n-r-s\le 1$).
For all $i,j\in\{r+1,\dots,r+s\}$, we have
$\Theta(x_i,y_j)=\Theta(y_i,x_j)=\Theta(y_i,y_j)=0$, whence
$$
   \theta(x_i,x_j) = \Theta(x_i+y_i,x_j+y_j)=\Theta(\alpha_i,\alpha_j)=0.
$$
Thus there is a subspace of $V$ of dimension $s$ on which $\theta$ vanishes,
and it remains to show that $s\ge n-1$
(that is (\ref{eq.transit}) with $k=1$). Suppose that $s<n-1$.
In our context this means that $r=1$ and $s=n-2$, i.e.,
$\alpha_1=x_1+y_1+z_1$ and $\alpha_n=y_n$ with $z_1\ne 0$ and $y_n\ne 0$.
We have $\Theta(x_1,y_n)=\Theta(y_1,y_n)=0\ne\Theta(z_1,y_n)$
(see the penultimate column of $B$), which contradicts the fact that
$\Theta(x_1+y_1+z_1,y_n)=\Theta(\alpha_1,\alpha_n)=0$.
\end{proof}


\subsection{Computation of the coindex of a symmetric form over $\Q$}\label{sect.Serre}

Let $Q$ be a non-degenerate symmetric $n\times n$ matrix over $\Q$.
The computation of $\beta(Q)$ is known for more than a century and
it is exposed (sometimes more sometimes less explicitly)
in all textbooks on quadratic forms.
For the reader's convenience we present it here in the form given in Serre's book
\cite{Ser}. In \cite{MH} (see also \cite[Appendix B]{Liv})
a computation of $\beta(Q)$ is presented in another form,
without explicit usage of the Hasse symbol.

Let $[a_1,\dots,a_n]$ be a diagonalization of $Q$ such that all the $a_i$'s are
square-free integers. Let $P$ be the set of prime divisors of $\prod a_i$
which do not divide $\gcd(a_1,\dots,a_n)$.
Then
$$
     \beta(Q) = \max\big\{\, \beta_\R(Q),\, \max_{p\in P\cup\{2\}}\beta_{\Q_p}(Q)\,\big\}
     \qquad\text{ and }\qquad
     \beta_\R(Q)=|\sigma|
$$
where $\sigma$ is the signature of $Q$.

The $\beta_{\Q_p}(Q)$ are computed as follows. Let
$\hasse=\hasse_p(Q)=\prod_{i<j}(a_i,a_j)_p$ be
the {\it Hasse symbol} (the ``invariant $\varepsilon$'' in \cite[Ch.~4, \S2]{Ser});
here $(\,,)_p$ is the Hilbert symbol.
Note that a different definition of the Hasse symbol (with $\prod_{i\le j}$ instead
of $\prod_{i<j}$) is used in \cite{Lev1}. Let $d=d_p(Q)$ be the {\it discriminant} of $Q$,
that is the class of $\det Q$ in $\Q_p^\times/(\Q_p^\times)^2$.
We define the {\it normalized discriminant} $\tilde d$ and {\it normalized Hasse invariant}
$\tilde\hasse$ as follows%
\footnote{Some authors define the discriminant of $Q$ to be $\tilde d$.}
$$
   \tilde d = (-1)^ld\in\Q_p^\times/(\Q_p^\times)^2,
   \qquad
   \tilde\hasse = (-1,d)_p^l(-1,-1)_p^{l(l+1)/2}\hasse,
   \qquad
   \text{where $l=\lfloor n/2\rfloor$.}
$$
Note that $(-1,-1)_p=(-1)^{p-1}$.

If $n=2l$, then
$$
  \beta_{\Q_p}(Q) = \begin{cases}
     0      &\text{if $\tilde d=1$ and $\tilde\hasse=1$,}\\
     2      &\text{if $\tilde d\ne 1$,}\\
     4      &\text{if $\tilde d=1$ and $\tilde\hasse=-1$.}
  \end{cases}
$$

If $n=2l+1$, then
$$
  \beta_{\Q_p}(Q) = \begin{cases}
     1      &\text{if $\tilde\hasse=1$,}\\
     3      &\text{if $\tilde\hasse=-1$.}
  \end{cases}
$$


\section{Isometric structures} \label{isom}

Let $\F$ be a field of characteristic zero (as above, the most interesting cases for
us are $\F$ is $\Q$, $\R$, or $\Q_p$).

\begin{definition}
An \emph{isometric structure over} $\F$ is a triple $(V,b,T)$ where $b$
is a non-singular symmetric bilinear form on a finite dimensional $\F$-vector
space $V$ and $T$ is an isometry of $(V,b)$, i.e. 
\begin{equation}
b(x,y)=b(Tx,Ty) \text{ for all } x,y \in V.                       \label{isome}
\end{equation}
\end{definition}

Let $\Delta_T(t)=\det(T-tI)$ be the characteristic polynomial of $T$.

If $Q$ and $T$ are $n\times n$ matrices, $T$ is nonsingular,
$\transpose{Q}=Q$, and $\transpose{T}QT=Q$, then the corresponding 
isometric structure on the coordinate space $\F^n$ is denoted by $(Q,T)$.


\subsection{The concordance group $\bar{\mathcal{G}}_\F^*$}\label{sect.isom.concord}

\begin{definition}
An isometric structure $(V,b,T)$ is \emph{admissible} if $\Delta_T(-1) \neq 0$.
\end{definition}

The \emph{sum} of two isometric structures is defined as 
$$
      (V,b,T) \oplus (V',b',T')=(V\oplus V',b \oplus b', T \oplus T').
$$

\begin{definition}
An isometric structure $(V,b,T)$ is \emph{metabolic} if $V$ contains a subspace
of dimension $\frac12\dim V$, totally isotropic for $b$ and invariant under $T$. 
Two isometric structures $(V_1,b_1,T_1)$ and $(V_2,b_2,T_2)$ are \emph{concordant}
if $(V_1,b_1,T_1) \oplus (V_2,-b_2,T_2)$ is metabolic.
\end{definition}

The concordance of admissible isometric structures is an equivalence relation.
Equivalence classes form an abelian group under orthogonal sum denoted
$\bar{\mathcal{G}}_\F^*$, and the neutral element is the class of
admissible metabolic structures \cite[\S6]{Lev1}.

Let $A$ be a non-singular admissible matrix.
Then the congruence class of the matrix $A+\transpose{A}$ and the similarity class
of $A^{-1}\transpose{A}$ are determined by the congruence class of $A$.
This implies that the congruence
class of $A$ determines a well-defined isometric structure
$(Q,T)=(A+\transpose{A}, A^{-1}\transpose{A})$.
Note that
\begin{equation} \label{DeltaA=det(A)DeltaT}
                              \Delta_A(t)=\det(A) \Delta_T(t).
\end{equation}
Let $\bar{\mathcal{G}}^\F_* $ be the subgroup of $\bar{\mathcal{G}}^\F$ of
concordance classes of admissible matrices $A$ such that $\Delta_A(-1) \neq 0$.

\begin{proposition}\label{prop.isom}
\cite[\S\S 8--9]{Lev1}   \label{iso}
The morphism 
\begin{align*}
   \bar{\mathcal{G}}^\F_* & \longrightarrow  \bar{\mathcal{G}}_\F^*, \qquad
      A \longmapsto   (A+\transpose{A}, A^{-1}\transpose{A}),
                  \qquad \det A\ne 0,
\end{align*}
is well-defined and it is an isomorphism.
\end{proposition}

The inverse is given by $(Q,T)\mapsto Q(I+T)^{-1}$
(see \cite[Thm.~3.5.16]{LivNaik} for a proof that this is indeed the inverse).
The proof of Proposition \ref{iso} follows mainly from the following lemma.

\begin{lemma}\label{lem.nonsing} {\rm\cite[Lemma 8]{Lev1}}
Every matrix $A$ with $\Delta_A(-1) \neq 0$ is S-equivalent (and then concordant
by Proposition~\ref{S=>conc}) to an admissible non-singular matrix.
\end{lemma}

The isomorphism in Proposition~\ref{prop.isom} can be defined in invariant terms
by $(V,\theta)\mapsto(V,b,T)$ where $\theta$ is a non-degenerate bilinear
form on a vector space $V$, $b=\theta+\transpose\theta$,
and $T:V\to V$ is such that $\theta(x,y)=\theta(y,Tx)$ for any $x,y\in V$.


\subsection{Decomposition of $\bar{\mathcal{G}}_\F$} \label{modules}

Given an admissible isometric structure $(V,b,T)$,
the $\mathbb{F}$-vector space $V$ can be viewed as
an $\mathbb{F}[t^{\pm 1}]$-module through the action of $t$, as follows:
$$
    t v= T(v) \text{ for all }  v \in V.
$$
Note that $V$ is a torsion module, annihilated by $\Delta_T(t)$. 
The \emph{$\delta$-primary component} of $V$, where  $\delta$ is an irreducible
factor of $\Delta_T$,  is defined as
\begin{equation}\label{eq.V.delta}
        V_\delta = \ker \delta(T)^m \ \text{ for a large } \ m.
\end{equation}
Recall that $t +1$ cannot be a factor of $\Delta_T$ since $(V,b,T)$ is admissible.

By \cite[Lemma 10]{Lev1}, if $\delta_i$ and $\delta_j$ are irreducible factors of $\Delta_T$
such that $\delta_i(t)$ and $\delta_j(t^{-1})$ are relatively prime,
then $V_{\delta_i}$ and $V_{\delta_j}$ are orthogonal. In particular,
$V_{\delta}$ and $V_{\bar\delta}$ are orthogonal if $\delta$ is not symmetric
(see the definition in Section~\ref{sect.propor}).
Note that an irreducible symmetric polynomial different from $t \pm 1$ is
palindromic of even degree. Thus there is an orthogonal decomposition
$$
    V= \Big(\bigoplus_{\delta} V_\delta\Big)\oplus\Big(\bigoplus_\lambda(V_\lambda\oplus V_{\bar\lambda})\Big)
$$
where the $\delta$ are symmetric and the $\lambda$ are non-symmetric factors of $\Delta_T$.
The restriction of $(b,T)$ to each summand $V_\lambda\oplus V_{\bar\lambda}$ is metabolic
(because $\lambda$ is orthogonal to itself and hence $b$ vanishes on $V_\lambda\oplus 0$).
As a direct consequence of these observations, we obtain that the isometric structure $(b,T)$
is metabolic if and only if its restriction to each of the factors $V_\delta$ is metabolic.
This fact can be equivalently formulated as follows.

\begin{proposition}\label{Levine.V_delta}
\cite[Lemma 11]{Lev1} The group $\bar{\mathcal{G}}_\F $ decomposes as a sum
$\bar{\mathcal{G}}_\F \simeq  \bigoplus_{\delta} \bar{\mathcal{G}}_\F^\delta$,
where $\delta$ runs over all irreducible symmetric polynomials and
$\bar{\mathcal{G}}_\F^\delta$ is the subgroup of concordance classes of isometric
structures whose characteristic polynomial is a power of $\delta$.
\end{proposition}

We also need the following Levine's result.

\begin{proposition} \cite[Lemma 12]{Lev1} \label{minimal}
If $\delta$ is an irreducible symmetric polynomial and $(V,b,T)$ has characteristic polynomial $\delta^k$ for some
 positive $k$, then  $(V,b,T)$ is metabolic or concordant to an isometric structure  $(V_0,b_0,T_0)$ with minimal polynomial $\delta$,
 i.e., such that $\delta(T_0)=0$.
 \end{proposition}


\subsection{Minimal isometric structures} \label{MinIS}

An isometric structure $(V,b,T)$ is called {\it minimal} if $\dim V$
is minimal in its concordance class.

\begin{lemma}\label{lem.min0}
Let $(V,b,T)$ be a minimal isometric structure.
Then $V$ does not have nontrivial isotropic $T$-invariant subspaces.
\end{lemma}

\begin{proof}
Suppose that $b$ vanishes on a $T$-invariant subspace $W\ne\{0\}$.
Without loss of generality we suppose that $W$  is generated by some vector $w$,
as a $\F[t^{\pm1}]$-submodule of $V$, and that the minimal polynomial which
annihilates $w$ is an irreducible factor $\delta$ of $\Delta_T$.
Let us choose a vector $w' \in V_\delta \setminus W$ such that $b(w,w')\ne 0$.
Such a vector exists because otherwise $b$ would vanish on $V_\delta$,
which is impossible since $b$ is non-singular and $V_\delta$ is an orthogonal
summand of $V$. Let $W'$ be the submodule generated by $w'$.
By Proposition~\ref{Levine.V_delta}, the minimal polynomial of $w'$ is $\delta$.
Hence $W \cap W'= \{ 0 \}$.

Let us show that $b|_{W \oplus W'}$ is non-singular.
Suppose that $u+u'$ is orthogonal to $W \oplus W'$ for some $u\in W$ and $u'\in W'$.
Then $(u+u')\perp W$ and $u\perp W$ (because $b|_W=0$), hence $u'\perp W$.
Since $W'$ is cyclic and $\delta$ irreducible, $W'$ is generated by any its non-zero
element. Therefore, if $u'$ where non-zero, we would have $W'\perp W$, which
contradicts the condition $b(w,w')\ne 0$.
Thus $u'=0$ and hence $u=u+u'$ is orthogonal to $W'$. Since $W$ is also cyclic,
the same arguments show that $u=0$.

Thus $b|_{W \oplus W'}$ is non-singular, hence $V$
splits into the orthogonal sum of two $T$-invariant subspaces
$(W\oplus W')\oplus(W \oplus W')^\perp$. Moreover,
the restriction of $b$ to $W\oplus W'$ is metabolic,
which contradicts the minimality of $\dim V$ in the concordance class.
\end{proof}

We do not use the following fact, but it seems to be interesting by itself
and it is an almost immediate consequence from the previous lemma.

\begin{proposition}\label{lem.min1}
If two minimal isometric structures are concordant, then they are isomorphic.
\end{proposition}

In other words, each class in $\mathcal{G}_\F^*$ contains a unique
minimal isometric structure up to isomorphism.

\begin{proof}
Let $(V_1,b_1,T_1)$ and $(V_2,b_2,T_2)$ be two concordant minimal isometric
structures and let $n=\dim V_1=\dim V_2$. Then
$(V,b,T):=(V_1\oplus V_2,\,b_1\oplus-b_2,T_1\oplus T_2)$ has a $T$-invariant
isotropic subspace $U$ of dimension $n$. Let $\pi_k:U\to V_k$, $k=1,2$, be the
restriction of the canonical projection $V\to V_k$. Then $\ker\pi_1$ is a
$T$-invariant isotropic subspace of $V_2$, hence $\ker\pi_1=0$ by
Lemma~\ref{lem.min0}, i.e., $\pi_1$ is an isomorphism of $\F$-vector spaces.
The same is true for $\pi_2$. Then it is straightforward to check that
$-\pi_2\circ\pi_1^{-1}$ is an isomorphism between the considered isomorphic structures.
\end{proof}


\section{Forms of coindex $1$}\label{sect.coind=1}
\subsection{The main idea}
Levine's theory \cite{Lev1} starts from the observation that if a Seifert matrix $A$
is metabolic then any totally isotropic subspace $U$ is $T$-invariant (following the
notations of \S\ref{isom}). Indeed $U$ and $T(U)$ are contained in the right orthogonal
complement $\transpose{U}_A$ of $U$, and then $T(U)=\transpose{U}_A=U$ since the condition
$\beta(A)=0$ (i.e. $A$ is metabolic) implies that the dimensions of all the three
subspaces coincide.

If $\beta(A)=1$, then the dimensions are no longer equal, but $U$ and $T(U)$ are
two hyperplanes in $\transpose{U}_A$, and $U_1=U\cap T(U)$ is a hyperplane in $U$.
Then we show that $U_2=T^2(U)\cap U$ is a hyperplane in $U_1$ etc.
In this way we construct a basis  whose first  elements generate $U$,
and the matrix of $T$ is the companion matrix of $\Delta_A$. Then the $T$-invariance
of $Q=A+\transpose{A}$  means that the entries of $Q$ verify 
 a system of linear equations which admits a unique solution.
We obtain that the isometric structure is concordant to a standard one,
determined uniquely by $\Delta_A$.
Unfortunately these arguments do not work in the case $\beta(A)>1$.

%
\subsection{Standard coindex-one matrices}\label{sect.def.coind1}
In this section $\F$ is a field of characteristic zero
(the most important case being $\F=\Q$).
Let $n \ge 0$ and let $P(t)$ be an antipalindromic polynomial
(see Section~\ref{sect.propor}) of degree $2n+1$ in $\mathbb{F}[t]$
such that $P(0)\ne 0$. Then $P$ can be written
$$
    P(t) = 
   - c_0 t^{2n+1} - c_1 t^{2n} - \dots -c_n t^{n+1} + c_n t^n + \dots + c_1 t + c_0,
  \qquad c_0\ne 0. 
$$
Consider the matrices
$$ 
 A_P= \begin{pmatrix}
                          0   &  \transpose{M}_{n+1} \\
                          M_n &  0
      \end{pmatrix}^{-1}
 \quad\text{ where }\quad
M_i= \begin{pmatrix}
 c_0 &  c_1   &   c_2  & \cdots &  c_{i-1} \\
     &  c_0   &   c_1  &  \dots &  c_{i-2} \\
     &        & \ddots & \ddots &  \vdots  \\
     &        &        &  c_0   &  c_1     \\
     &        &        &        &  c_0     \end{pmatrix}
$$
(here the $M_i$ are upper triangular and then
the matrix $A_P^{-1}$ is ``butterfly-shaped'').
It is easy to see that $\beta(A_P)=1$ (see Example~\ref{example.coind1}).

\begin{definition}\label{def.standard} The
{\it standard coindex-one matrix associated with $P$} is $A_P$. We denote it by $\std(P)$.
\end{definition}

Let $T_P$ be the companion matrix of $P$, i.e.,~the matrix of the multiplication by $t$
in the $\F[t]$-module $\F[t]/(P)$ in the basis $(1,t,\dots,t^{2n})$.
A straightforward computation shows that
$T_P(A_P^{-1})\transpose{{}}=A_P^{-1}\lower-2pt\hbox{\mathstrut}$, hence
$A_P^{-1}\transpose{A}_P = T_P$ and $(A_P+\transpose{A}_P,T_P)$ is the
corresponding isometric structure.

Since $T_P$ is the companion matrix of $P$, we have $\Delta_{T_P}(t)=P(t)/c_0$.
It is also clear that
$\det A_P^{-1}=\det M_n\det M_{n+1}=c_0^{2n+1}$.
Hence, by (\ref{DeltaA=det(A)DeltaT}), we obtain
\begin{equation}\label{Delta.AP}
               \Delta_{\std(P)}(t)=  \Delta_{A_P}(t) = c_0^{-2(n+1)} P(t).
\end{equation}

 \begin{proposition}\label{prop.coind1}
 Let $P$, $A_P$, $T_P$ be as above.
 Let $(Q,T_P)$ be an isometric structure
 such that the left upper $(n\times n)$-submatrix of $Q$ is zero.
 Then $Q=\lambda(A_P + \transpose{A}_P)$, $\lambda\in\F$.
 \end{proposition}

 \begin{proof}
 Without loss of generality we may assume that $c_0=1$. Then the last column of $T_P$ is
 $(1,c_1,\dots,c_n,-c_n,\dots,-c_1)\transpose{}$.
 We have $\transpose{T}_P Q T_P = Q$. Hence $Q$ is a Toeplitz matrix
 (i.e., $Q_{i,j}=Q_{i+1,j+1}$ for each $i,j<2n+1$) with the first line
 starting with $n$ zeros: $(0,\dots,0,x_0,x_1,\dots,x_n)$.
 The first $n$ entries of the last column of 
 $Q - \transpose{T}_P Q T_P$ provide the following equations,
 which allow us to express $x_1,\dots,x_n$ via $x_0$:
\begin{align*}
c_n x_0 + c_{n-1} x_1 + \dots + c_2 x_{n-2} + c_1 x_{n-1} + x_n\;\;\;   &= 0,\\
          c_{n-1} x_0 + \dots + c_2 x_{n-3} + c_1 x_{n-2} + x_{n-1}     &= 0,\\
                        \vdots \qquad \qquad & \\
                                c_2 x_0\;\;  +\; c_1 x_1\;\; +\; x_2\;\;\; &= 0,\\ 
                  -x_0 \qquad\qquad\qquad    +\; c_1 x_0\;\; +\; x_1\;\;\; &= 0.
\end{align*}
This means that $Q$ is determined by $P$ up to a scalar factor whence the result.
\end{proof}


\subsection{The main result}

If $A$ is a matrix of odd dimension, then
\eqref{sym.Delta} implies that $\Delta_A$ is antipalindromic of odd degree
and hence the matrix $\std(\Delta_A)$ is well defined. 
Moreover, the parity of degree and the property to be (anti)palindromic are evidently
invariant under the Fox-Milnor equivalence, hence $\std(P)$ is well defined if $P\eFM\Delta_A$.

\begin{theorem} \label{coindex1}
Let $A$ be a square matrix of odd dimension such that $\Delta_A(-1) \neq 0$ and
$\beta(A)=1$. Then for any polynomial $P\in\F[t]$ such that $\Delta_A(t)\eFM P(t)$,
the matrix $A$ is concordant to the standard coindex-one matrix $\std(P)$ associated with $P$.
\end{theorem}

The proof of Theorem \ref{coindex1} can be mainly brought back
to a particular case stated in the following lemma which will be proven below. 

\begin{lemma} \label{part}
Let $A$ be a matrix of odd dimension such that $\Delta_A(-1)\ne 0$.
Let $P= \prod_{i=1}^s\delta_i$ be a product of polynomials,
where the $\delta_i$'s are pairwise distinct symmetric irreducible.
Suppose that  $\Delta_A(t)\eFM P(t)$.
If $\beta(A)=1$ then $A$ is concordant to $\std(P)$. 
\end{lemma}

\begin{proof}[Proof of Theorem \ref{coindex1}]
Let $A$ and $P$ be as in Theorem \ref{coindex1}. It is clear that $\Delta_A\eFM P_0$ where
$P_0$ is a product of pairwise distinct irreducible symmetric polynomials
(because any symmetric polynomial $F$ with $F(0)\ne 0$ is Fox-Milnor equivalent to such a product).
Let $B=\std(P)$. By (\ref{Delta.AP}) we have $\Delta_B\eFM P$, hence
$\Delta_B\eFM P\eFM\Delta_A\eFM P_0$. Thus $\Delta_A\eFM P_0$ and $\Delta_B\eFM P_0$.
Moreover, $\beta(A)=\beta(B)=1$, hence both matrices $A$ and $B$ are concordant to
$\std(P_0)$ by Lemma~\ref{part} and thus $A$ is concordant to $B=\std(P)$.
\end{proof}
 
In order to prove Lemma \ref{part}, we use the following Lemma \ref{concordance},
which states that if Lemma \ref{part} holds for a nonsingular matrix $A$ then it holds
for any nonsingular matrix concordant to $A$. 

\begin{lemma} \label{concordance}
Let $A$ and $B$ be nonsingular matrices such that $\Delta_A\eFM P$ and $\Delta_B\eFM Q$
where $P$ and $Q$ are products of pairwise distinct symmetric irreducible polynomials.
If $B$ is concordant to $A$ and $A$ is concordant to $\std(P)$, then $B$ is concordant to $\std(Q)$.
\end{lemma}

\begin{proof}
Since $A$ and $B$ are concordant and nonsingular, one has $\Delta_A\eFM \Delta_B$
by Lemma \ref{alexconc}.  Hence $P\eFM\Delta_A\eFM\Delta_B\eFM Q$. Since $P$ and $Q$ are
products of pairwise distinct symmetric polynomials, it follows that there exists
$c\in \F^\times$ such that $Q(t) = c^2 P(t)$, whence $\std(Q)=c^2 \std(P)$.
In particular, $\std(P)$ and $\std(Q)$ are concordant and the result follows.
\end{proof}

\begin{proof}[Proof of Lemma \ref{part}]
Since $A$ is admissible,  $A$ is S-equivalent to a non-singular matrix $\bar{A}$ by
Lemma~\ref{lem.nonsing}. Then $\Delta_A=\Delta_{\bar{A}}$ (see \S\ref{sect.alex}) and $A$
is concordant to $\bar A$, hence the concordance between $\bar A$ and $\std(\Delta_{\bar A})$
implies the concordance between $A$ and $\std(\Delta_A)$. Due to this observation, we assume
from now on that $A$ is nonsingular. By Lemma \ref{concordance}, we may also assume
that the size of $A$ is minimal in its concordance class.

Let $2n+1$ be the size of $A$ and $(V,b,T)$ be the isometric structure related to $A$, i.e.,
$V=\F^{2n+1}$, $T=A^{-1}\transpose{A}$, and $b(x,y)=\transpose{x}(A+\transpose{A})y$.
Since the size of $A$ is minimal, the isometric structure is minimal.
Then it decomposes according to the irreducible factors of $\Delta_A$ as
$V=\bigoplus_{i=1}^s V_{\delta_i}$, where the $\delta_i$ are symmetric and irreducible
and, moreover, $V_{\delta_i}=\ker\delta_i(T)$ for all $i$
(see Propositions \ref{Levine.V_delta} and \ref{minimal}).
Let $U$ be an $n$-dimensional totally isotropic subspace of $V$.
For all $x$ and $y$ in $U$, we have
$$
     \transpose{y}\! A \, (Tx)= \transpose{y}\!
     A(A^{-1}\transpose{A}) \, x= \transpose{y}\!\transpose{A}
     x= (\transpose{x}\! A y\transpose{)}=0.
$$
Hence $T(U)$ is a subspace of the right orthogonal
supplementary space of $U$ with respect to $A$:
$$
       T(U)\subset U_{A}^\perp= \{ z \in V \mid
       \transpose{y}\! A \, z=0 \text{ for any } y \in U \}.
$$
Let us set inductively:
\begin{equation}                                           \label{0}
                  U_0=U, \quad U_{i+1}=T(U_i) \cap U.
\end{equation}
We have $U_1\subset U_0$ and, for any $i\ge 1$, the condition $U_i\subset U_{i-1}$
implies
$$
         U_{i+1}  =  T(U_i)\cap U  \subset  T(U_{i-1})\cap U  =  U_i.
$$
Hence, by induction, we have $U=U_0 \supset U_1 \supset U_2 \supset \dots$.
Let $k$ be such that
$$
            n= \dim U_0 > \dim U_1 > \dots > \dim U_k= \dim U_{k+1}.
$$
The fact that $A$ is non-singular implies that $\dim U^{\perp}_A= \dim V- \dim U=n+1$.
This fact combined with $U \subset U^\perp_{A}$, $T(U) \subset U^\perp_A$,
and $\dim U= \dim T(U)=n$ yields
\begin{equation}                                              \label{1}
                     \dim U_1= \dim(U \cap T(U)) \geq n-1.
\end{equation}
Since $T$ induces an isomorphism $U_{i-1}/U_i \rightarrow T(U_{i-1})/T(U_i)$,
we have also
\begin{equation}                                             \label{2}
      \dim U_{i-1}/U_i
    \geq \dim\big(T(U_{i-1}) \cap U\big)/\big(T(U_i)\cap U)
      =\dim U_i/U_{i+1}.
\end{equation}
By combining (\ref{1}) and (\ref{2}), we obtain
\begin{equation}                                              \label{3}
    \dim U_i = n-i \ \ \text{ for } \ \
                0 \leq i \leq k.
\end{equation}
By definition, $U_{k+1}\subset T(U_k)$ and $\dim U_{k+1} = \dim U_k = \dim T(U_k)$,
hence $U_{k+1}=T(U_k)$. We also have $U_{k+1}=U_k$. Thus 
$U_k$ is $T$-invariant whence $U_k=\{0\}$ by Lemma~\ref{lem.min0}. Hence
$k=n$ by (\ref{3}). Choose $u_n \in U_{n-1}\setminus\{0\}$ and set
$$
     u_i= T^{-(n-i)}(u_n) \ \ \text{ for } \ \ i=1,\dots ,n-1.
$$
By (\ref{0}) we have $T^{-1}(U_{i+1}) = U_i\cap T^{-1}(U)\subset U_i$ whence,
by induction,
$u_i \in U_{i-1} \setminus U_i$ for $ i=1,\dots ,n$. Hence $(u_1,\dots,u_n)$
is a basis of $U$.
 
Let $W$ be the $\F[t^{\pm 1}]$-submodule of $V$ generated by $u_1$.
Let us show that $\dim W=2n +1$.

\smallskip
\emph{Case 1.} Suppose that $\dim W = m < 2n$.
Since $b|_U=0$ we have $b(u_n,u_i)=0$ for $i=1,\dots,n$. If $n<i\le m$, then
$1< i-n+1 \le n$, hence
$$
    b(u_n,u_i) = b(T^{n-1}u_1,T^{n-1}u_{i-n+1}) = b(u_1,u_{i-n+1}) = 0.
$$
Thus $b(u_n,w)=0$ for each $w\in W$. Since $W$ is generated by $u_n$ as
an $\F[t^{\pm1}]$-module, it follows that $b|_W=0$, which contradicts
Lemma~\ref{lem.min0}.

\smallskip
\emph{Case 2.} Suppose that $\dim W = 2n$.
There is an orthogonal sum decomposition
$W=\oplus_i W_{\delta_i}$ where $W_{\delta_i} = W\cap V_{\delta_i}$.
The module $W$ is cyclic, generated by the single element $u_1$.
Hence for all $i=1,\dots,s$, $W_{\delta_i}=\ker\delta_i(T|_W)$ is cyclic as well and 
$\dim W_{\delta_i} = \deg\delta_i$. Since $\delta_i$ is symmetric for all $i \neq i_0$,
it follows that $\deg \delta_i$ is even for all $i \neq i_0$. Hence $t-1$ is the only
odd degree factor of $\Delta_A$ and then of $Q$. From the hypothesis $\dim W=2n$
we obtain $W_{t-1}=0$, and then the orthogonal decomposition $V=W\oplus V_{t-1}$.
The restriction $A|_W$ is metabolic, hence $b|_W$ is metabolic, which contradicts the
minimality of $\dim V$ in the concordance class of $(V,b,T)$.

We have then shown that $\dim W = 2n+1$. Hence $V=W$ and $(u_1,\dots,u_{2n+1})$ is a
basis of $V$. In this basis, the matrix of $T$ is the companion matrix of $\Delta_T$
and the matrix of $b$ has a zero $n\times n$ submatrix at the left upper corner.
Since the $\F[t]$-module $V$ is cyclic, the characteristic polynomial of $T$ coincides
with its minimal polynomial, proportional to $P$.
Then by Proposition~\ref{prop.coind1} the matrix $A+\transpose{A}$ is proportional to
$B+\transpose{B}$, where $B:=\std P$. Hence $B=\mu A$ with $\mu\in\F^\times$. We also have
$\Delta_B\eFM P\eFM\Delta_A$ (see (\ref{Delta.AP}) for the first equivalence),
thus $A$ is concordant to $B=\std(P)$ by Lemma~\ref{lem.propor}.
\end{proof}


\subsection{The $(t-1)$-primary component}\label{sect.t-1}
By Theorem~\ref{coindex1}, determining whether $\beta(A)=1$ or not is reduced to 
decide if $A$ is concordant to a given standard matrix (associated with $\Delta_A$).
This can be done by comparing their isometric structures.
By Proposition~\ref{Levine.V_delta}, two isometric structures are concordant if
and only if all their primary components are concordant.
The study of the $(t-1)$-primary component could \emph{a priori} be sufficient in
some cases to prove that $\beta(A)>1$, see Example~\ref{ex.(t-1)}. However
(see the discussion after Example~\ref{ex.(t-1)}), the most interesting case
would be when $A$ is a Seifert matrix of a two-component link. The following
proposition shows in fact that the $(t-1)$-primary component
does not give any restriction in this case. 

\begin{proposition} \label{prop.linking} Let $L=k_1\cup k_2$ be a two-component
link such that $\Delta_L(-1)\ne 1$. Let $A_L$ be a Seifert matrix of $L$ and
$B_L=\std(\Delta_L)$ be the standard coindex-one matrix associated with $\Delta_L(t)$.
Then the $(t-1)$-primary components of the isometric structures corresponding
to $A_L$ and $B_L$ are concordant.
\end{proposition}

The rest of the section is devoted to the proof of Proposition \ref{prop.linking}.
 We need two preliminary results.

\begin{proposition}\label{prop.(t-1)}
Let $A$ be a matrix such that $\Delta_A(-1)\ne 0$. If $1$ is a simple root of
$\Delta_A(t)$. Then the $(t-1)$-primary component of the isometric structure
of $A$ is concordant to  $(V,b,\id_V)$, where $\dim V=1$ and the value of $b$
on a generator of $V$ is $-2\Delta'_A(1)$.
\end{proposition}

Notice that if $k_1\cup k_2$ is a two-component link with Seifert matrix $A$,
then $\Delta'_A(1)=-\lk(k_1,k_2)$ (cf.~\cite{hoso}).

\begin{proof}
$A$ is congruent to $(a)\oplus B$ for some matrix $B$. Hence
$\Delta_A(t)=a(1-t)\Delta_B(t)$ and $\Delta'_A(1)=-a\Delta_B(1)=-a\det(B-\transpose{B})$.
It remains to note that the determinant of any skew-symmetric matrix is a square.
\end{proof}

\begin{lemma}\label{lem.(t-1)}
Let $A$ be a matrix with $\Delta_A(-1)\ne 0$ and
$\dim\ker(A-\transpose{A})=1$.
Then $A$ is concordant to a matrix such that the $(t-1)$-primary component
of its isometric structure is one-dimensional.
\end{lemma}

\begin{proof}
By Lemma~\ref{lem.nonsing} any admissible matrix is S-equivalent (over $\F$) to a
non-singular matrix. One easily shows that $\dim\ker(A-\transpose{A})$ is invariant
under S-equivalence (see, however, Remark~\ref{rem.(t-1)}).
Thus we may assume that $A$ is non-singular.

Let $(V,b,T)$ be the $(t-1)$-primary component of the corresponding isometric structure.
According to Proposition \ref{minimal}, $(V,b,T)$ is either metabolic or concordant to
an isometric structure with minimal polynomial $t-1$. This fact is proven in \cite{Lev1}
by induction: if the minimal polynomial of $T$ is $\delta^a$ with $a>1$, then it is shown
that $(V,b,T)$ is concordant to $(V',b',T')$ and $\dim V'<\dim V$, where
$V'=W^\perp/W$, $W=\delta(T)^{a-1}V$, and $b'$ and $T'$ are induced by $b$ and $T$.

In order to prove the lemma, it is enough to show that the condition $\dim\ker(T-I)=1$
is preserved during the inductive process. Indeed, we have $A(T-I)=\transpose{A}-A$,
hence $\dim\ker(T-I)=\dim\ker(A-\transpose{A})=1$. Therefore, the matrix of $T$ in a
suitable base $v_1,\dots,v_m$ of $V$ is a Jordan block, i.e., $T_{i,i}=T_{i,i+1}=1$
for all $i$. The $T$-invariance of $b$ yields a system of linear equations for the
entries of the Gram matrix $Q$ of $b$. By solving some of them, we obtain that $Q$
is lower triangular with respect to the secondary diagonal (i.e. $Q_{ij}=0$ when 
$i+j<m$) and the entries on that diagonal are $Q_{i,m-i}=(-1)^ic$ with $c\in\F$;
we may set $c=1$ because $Q$ is non-singular. Since $Q$ is symmetric, this fact implies,
in particular, that $m$, and hence the dimension of $A$, is odd.
Thus the subspaces $W$ and $W^\perp$ are spanned by $v_1$ and by $(v_1,\dots,v_{m-1})$
respectively. If $m>1$, then $\dim V'=\dim V-2$ and $T'$ is again a Jordan block, thus
$\dim\ker(T'-I)=1$.

By induction we reduce the lemma to the case where the minimal polynomial of $T$ is $t-1$
and the dimension of $\ker(T-I)$ is still equal to $1$. In this case the $(t-1)$-component
coincides with $\ker(T-I)$ and hence it is one-dimensional.
\end{proof}

\begin{remark} \label{rem.(t-1)}
In general the dimension of $\ker(A-\transpose{A})$ is not invariant under concordance.
For example, for $A=\left(\smallmatrix 0&1\\1&0\endsmallmatrix\right)$, we have
$A-\transpose{A}=0$ whereas $A$ is metabolic. Note that this matrix is realizable
as a Seifert matrix of the 3-component Hopf link $L$ where the orientation of one 
component is reversed, and we have $\chi_s(L)=1$.
\end{remark}

\begin{proof} [Proof of Proposition \ref{prop.linking}] In virtue of Proposition~\ref{prop.(t-1)}
and Lemma~\ref{lem.(t-1)}, it is enough to show that the kernels of $A_L-\transpose{A}_L$ and
$B_L-\transpose{B}_L$ are one-dimensional.

Indeed, the module $\F[t]/\Delta_L(t)$ is cyclic and every primary component is cyclic.
Then $\dim\ker(T-I)=1$, where $T=B_L^{-1}\transpose{B}_L$. Since
$\ker(B_L-\transpose{B}_L)=\ker(T-I)$, it follows that $\dim\ker(B_L-\transpose{B}_L)=1$.

Furthermore, $A_L-\transpose{A}_L$ is the matrix of the intersection form on a Seifert
surface of $L$. Its corank is $1$ because $L$ has two components.
Thus $\dim\ker(A_L-\transpose{A}_L)=1$.
\end{proof}


\subsection{ Livingston's theorem }

Recall that the {\it discriminant} of a polynomial $P(t)=\sum a_k t^k\in\F[t]$ of
degree $n$ is $\Disc P=a_n^{2n-2}\prod_{i<j}(\alpha_i-\alpha_j)^2$, where
$\alpha_1,\dots,\alpha_n$ are the roots of $P$ in an extension of $\F$.
This is a polynomial in $a_0,\dots,a_n$ with integral coefficients.

\begin{theorem} \cite[Theorem 5.1]{Liv}, \cite[Theorem 6.4.2]{LivNaik} \label{theorem.living}
\footnote{We corrected misprints in \cite{Liv} in the definitions of $\Disc$ and
$\tilde\Delta_A$ (denoted in \cite{Liv} by $\bar\Delta_A$).}
Let $A$ be a non-singular matrix with entries in $\Z$, such that $\Delta_A(1)\Delta_A(-1)\ne 0$.
Let $(Q,T)=(A+\transpose{A},A^{-1}\transpose{A})$ be the isometric structure of $A$.
Suppose that all irreducible symmetric factors of $\Delta_T(t)$ have even exponent.
Then for any prime $p$ that does not divide $2\det(A)\,\Disc\tilde\Delta_A$, the image of
$(Q,T)$ in $\bar{\mathcal G}^*_{\Q_p}$ is zero, where $\tilde\Delta_A(t)$ denotes the
product of all the distinct irreducible factors of $\Delta_A(t)$ in $\Z[t]$.
\end{theorem}

This theorem is formulated in \cite[Theorem 5.1]{Liv} for Seifert matrices of knots, however
its proof goes in our more general setting. The corresponding analog of \cite[Lemma 5.2]{Liv}
(which is used in the proof of \cite[Theorem 5.1]{Liv}) can be also derived from the following fact.

\begin{lemma} \label{lem.liv}
Let $P(t)$ be a palindromic polynomial of even degree with coefficients in a ring $\Lambda$.
Then there exists $\lambda\in\Lambda$ such that $\Disc P = \lambda P(1)P(-1)$.
\end{lemma}

\begin{proof} It is enough to assume that $\Lambda=\Z[x_0,\dots,x_n]$ and $P=\sum_{k=0}^{2n}x_k t^k$
with $x_k=x_{2n-k}$ for $k>n$. Then $\Disc P$ and $P(\pm1)$ are polynomials in $x=(x_0,\dots,x_n)$ that we
denote by $D(x)$ and $P_{\pm1}(x)$.
It is clear that $P_{\pm 1}(x)$ is irreducible in $\Lambda$ and in $\bar\Lambda=\C[x_0,\dots,x_n]$.
Furthermore, $P_{\pm 1}(z)=0$ implies $D(z)=0$ for any $z\in\C^{n+1}$. Then by Hilbert's Nulstellensatz
$D$ is divisible by $P_{\pm1}$ in $\bar\Lambda$ and hence in $\Lambda$.
\end{proof}

\begin{remark} By computations similar to those in the proof of \cite[Lemma 5.2]{Liv},
one can show that the coefficient $\lambda$ in Lemma \ref{lem.liv} is equal to $(-1)^n \Disc S$,
where $S$ is the polynomial such that $P(t)=t^n S(t+t^{-1})$.
\end{remark}

\if01{ 
n=5
Q=Sum[a[k] x^k,{k,0,n}]
P = Expand[t^n Q/.x->t+1/t]
P1 = P/.t->1
P2 = P/.t->-1
DQ=Factor[Discriminant[Q,x]];
DP=Factor[Discriminant[P,t]];
Factor[DP/DQ^2/P1/P2]
}\fi 

\begin{remark} When completing details in the proof of Lemma~\ref{lem.liv}, one should take into account
that setting $a_n=0$ in the discriminant of $\sum_{k=0}^n a_kt^k$ (considered as a polynomial of degree $n$)
does not give the discriminant of $\sum_{k=0}^{n-1} a_kt^k$ (considered as a polynomial of degree $n-1$).
\end{remark}

\begin{remark} In the proof of \cite[Theorem 5.1]{Liv} (p.1701, l.1; here $Q$ instead of $Q^\delta$
is a misprint), the assertion that $\det Q^\delta$ is a square follows from
\cite[Theorem 2.6(3)]{Liv}, which is \cite[Lemma 7(c)]{Lev1}.
A proof of \cite[Thm.~C.1]{Liv} can be found in \cite[\S 14.6, Exercise 30]{DF}.
\end{remark}


\subsection{An algorithm to decide if $\beta(A)=1$}\label{sect.algo}
 
Let $A$ be a square matrix  of odd dimension, with entries in $\mathbb{Q}$, such that $\Delta_A(-1)\neq 0$.
In particular $A$ is admissible. We use Levine's theory combined with the characterization of
Theorem \ref{coindex1} to provide an algorithm deciding whether $\beta(A)=1$. 
Indeed, one has $\beta(A)=1$ if and only if $A$ is concordant (over $\Q$) to the standard
coindex-one matrix $\std(\Delta_A)$ associated with $\Delta_A$ (or in fact to any standard
coindex-one matrix associated with a polynomial FM-equivalent to $\Delta_A$).

Let $M_0=A\oplus-\std(\Delta_A)$. It is an admissible matrix of even dimension and $\beta(A)=1$
if and only if $M_0$ is metabolic.
Following the proof of Lemma~\ref{lem.nonsing} due to Levine, we transform $M$ into a non-singular
matrix $M$, S-equivalent to $M_0$.
   
Let $(Q,T)=(M+\transpose M,M^{-1}\transpose M)$ be the isometric structure of $M$. Since
$\Delta_M(-1) \neq 0$, the isometric structure $(Q,T)$ is admissible, thus $M$ is metabolic
if and only if the class of $(Q,T)$ in $\bar{\mathcal{G}}_\Q^*$ is zero (see Proposition \ref{prop.isom}).
Decompose  the class of $(Q,T)$ as a sum of classes of $(Q_{t-1},T_{t-1})$ and $(Q',T')$,
where $\Delta_{T_{t-1}}$ is a power of $t-1$ and $\Delta_{T'}(-1)\ne 0$.
We need to check if each of these two classes is zero in $\bar{\mathcal{G}}_\Q^*$.

In order to check if the class of $(Q_{t-1},T_{t-1})$ is zero in $\bar{\mathcal{G}}_\Q^*$,
it is enough to check if the symmetric form $Q_{t-1}$ is metabolic (for this we refer to \S\ref{sect.Serre}).

In order to check if the class of $(Q',T')$ is zero in $\bar{\mathcal{G}}_\Q^*$, we directly apply
the algorithm derived from Levine's and Livingston's papers \cite{Lev1} and \cite{Liv}.
In our setting it reduces to the following steps.

Let $M'$ be a matrix with integer entries which is proportional to $Q'(I+T')^{-1}$ (cf. \S\ref{sect.isom.concord}).
Due to \cite[Proposition 17]{Lev2} combined with Theorem~\ref{theorem.living} it is enough to check if
the image of $(Q',T')$ is zero in $\bar{\mathcal{G}}_\R^*$ and in $\bar{\mathcal{G}}_{\Q_p}^*$ for
all prime $p$ dividing $2 \det(M') \text{Disc}\tilde{\Delta}_{M'}$.
The vanishing of the image in $(Q',T')$ in $\bar{\mathcal{G}}_\R^*$ is easy to check since it is equivalent
to the vanishing of the signature function $\sigma_{M'}$.

In order to check whether the  image of $(Q',T')$ is zero in $\bar{\mathcal{G}}_{\Q_p}^*$, we
factorize $\Delta_{T'}(t)$ over $\Q_p$ and, for each irreducible symmetric factor $\delta$, 
we compute the restriction $(V_\delta,b_\delta,T_\delta)$ of the isometric structure $(Q',T')$ to the primary
component $V^\delta$ (see \eqref{eq.V.delta}) and
check whether $b_\delta$ is metabolic (we again refer to \S\ref{sect.Serre}).


\section{Link invariants and surfaces in the four-ball} \label{linkinv}

Recall that a Seifert surface is a compact, connected, oriented surface $S$ embedded in $S^3$ whose
boundary is a given oriented link. Let $H_1(S)$ be the first homology group,
with rational coefficients.
 
\begin{definition} \label{sequiv}
Let $L$ be an oriented link and $S$ be a Seifert surface of $L$,
with Seifert matrix $A$ in some basis of $H_1(S)$. We define the {\it coindex} of $L$
by setting
$$
    \beta(L):=\beta(A).
$$
It depends only on the isotopy class of $L$ and does not depend on the choice of $S$
(this fact follows from the invariance by S-equivalence; see Proposition~\ref{S-equiv}).
\end{definition}
 
In \S\ref{coind(link)} we show that the coindex is invariant by concordance for links with non-zero Alexander polynomial.
In \S\ref{bound.chi_s} we define the slice Euler characteristic $\chi_s$ of a link and prove that $\chi_s(L)\le 1-\beta(L)$.
In \S\ref{alex.sign.null} we discuss a relation of this bound with classical bounds of the slice Euler characteristic.


\subsection{Invariance of $\beta(L)$ under the topological link concordance}\label{coind(link)}
 Two oriented links are {\it concordant} if there exists a properly embedded collection of
 locally flat cylinders
 $S^1 \times I$ in $S^3 \times I$, intersecting $S^3 \times 0$ and $S^3 \times 1$
 at the given links.

This subsection is devoted to a proof of Theorem~\ref{main1} about
the invariance of $\beta(L)$ by (topological) concordance
under the assumption $\Delta_L\ne 0$.
We start with the following two elementary lemmas.

\begin{lemma} \label{linalg}
Let $E,F,$ and $G$ be $\mathbb{Q}$-vector spaces of finite dimension. Consider linear
maps $f: E \rightarrow F$ and $g: F \rightarrow G$. Then, the following holds
$$
    \dim \ker(g \circ f) \ge \dim E - \dim \operatorname{im} g.
$$
\end{lemma}

\begin{proof} Indeed,
$\dim \ker(g \circ f)
      = \dim E - \dim \operatorname{im}(g\circ f)
    \ge \dim E - \dim \operatorname{im} g$.
\end{proof}

\begin{lemma} \label{lagr}  
Let $\Sigma$ be a compact connected oriented surface and $M$ be a compact connected
three-manifold with $\partial M=\Sigma$. Let $i_*: H_1(\Sigma) \rightarrow H_1(M)$ be the
map induced by the inclusion in homology with rational coefficients.
Then $\dim \ker i_* = \dim\operatorname{im}i_*=\frac{1}{2} \dim H_1(\Sigma)$.
\end{lemma}

\begin{proof}
Consider the following commutative diagram, whose rows come from the long exact sequence
in rational (co)homology of the pair $(M,\Sigma)$ and the vertical arrows follow from
Poincar\'e duality. 
\[
\begin{tikzcd}
   H_2(M,\Sigma) \rar["\partial"] \dar & H_1(\Sigma)  \rar["i_*"] \dar & H_1(M)  \\
   H_1(M)^*      \rar["i^*"]           & H_1(\Sigma)^*                & 
\end{tikzcd}
\]
Hence $\dim\ker i_* = \rank\, \partial = \rank\, i^*=\rank\, i_*$
and the result follows from the rank theorem applied to $i_*$.
\end{proof}

\begin{proposition}\cite{Kaw}\label{inv.Delta=0}
If a link $L$ is concordant to $L'$ and $\Delta_L\ne 0$, then $\Delta_{L'}\ne 0$.
\end{proposition}

In \cite{Kaw} this result is proven for PL-concordance, but the proof easily
extends to the locally flat case. Otherwise, this fact follows from
\cite[Theorem~4.2]{NP}.

\begin{theorem} \label{main1}
 Let $L$ be a link in $S^3$ with $\Delta_L \neq 0$.
 If  $L'$ is concordant to $L$ then
 $$
                             \beta(L)=\beta(L').
 $$
 \end{theorem}

\begin{proof}
Let $C \subset S^3 \times I$ be the (locally flat) concordance with boundary
$L  \subset S^3 \times \{0 \}$ and  $L' \subset S^3 \times \{ 1 \}$.
Consider Seifert surfaces $S$ and $S'$ for $L$ and $L'$,
and let $\theta$ and $\theta'$ be the corresponding Seifert forms.
One can show that  there exists a compact connected oriented $3$-submanifold
$M$ in $S^3 \times I$, with boundary
$$
     \partial M = S \cup_L C \cup_{L'} S'.
$$
The proof of this fact uses the same arguments as Lemma \ref{cobound2}.
Let $h=\text{genus}(\partial M)$ and let $K$ be the kernel of the map
$i_* \circ j_*$ induced by the inclusions, with rational coefficients:
$$
   H_1(S \sqcup S') \simeq H_1(S) \oplus H_1(S') \stackrel{j_*}
   \longrightarrow H_1(\partial M) \stackrel{i_*} \longrightarrow H_1(M).
$$
By computing the Euler characteristics, we obtain
$\dim H_1(S\sqcup S')=\dim H_1(\partial M)=2h$. By Lemma~\ref{lagr} we also
have $\dim\operatorname{im}i_*=\frac12\dim H_1(\partial M)=h$.
Hence, by Lemma \ref{linalg},
$$
   \dim K \ge \dim H_1(S\sqcup S')-\dim\operatorname{im}i_*
          = 2h - h = h.
$$
 
Consider the bilinear form $\Theta= \theta \oplus (- \theta')$ on
$H_1(S) \oplus H_1(S')$ and let $x\oplus x'$ and $y \oplus y'$ be two elements of $K$. 
Let $\alpha,\beta$ be 2-chains in $M$ such that $\partial\alpha=x\oplus x'$ and
$\partial\beta=y \oplus y'$, and $\beta_+$ denote the push of $\beta$
from $M$ in the positive normal direction in $S^3 \times I$.
Then $\Theta$ is computed by the intersections in $S^3 \times I$:
$$
  \Theta(x\oplus x',y\oplus y')= \theta(x,y) - \theta'(x',y')= \alpha\cdot\beta_+.
$$
Since $\alpha\subset M$ and $\beta_+ \subset (S^3 \times I) \setminus M$,
these chains are disjoint and $\Theta(x\oplus x',y\oplus y') = 0$.
Hence $K$ is totally isotropic and $\dim K\ge h$, which means that
$\Theta$ is metabolic.
Thus $\theta'$ is concordant to $\theta$. Since $\Delta_L$ and $\Delta_{L'}$
are non-zero (see Proposition~\ref{inv.Delta=0}), the Seifert matrices are
admissible and the result follows from Proposition~\ref{algconc}.
\end{proof} 
 

\subsection{Bounding the slice Euler characteristic via the coindex}\label{bound.chi_s}
The \emph{slice Euler characteristic} of an oriented link $L$ is defined as
$\chi_s(L)= \max_{F} \chi(F)$,
where the maximum 
is taken over all locally flat oriented surfaces with boundary $L$,
without closed components in $B^4$ and transverse to $S^3$. 
We show in Theorem \ref{main2} that the coindex provides an upper bound for $\chi_s$.

\begin{theorem} \label{main2}
Let $L$ be a link in $S^3$. Then the following inequality holds
$$ \chi_s(L) \leq 1 - \beta(L).$$
\end{theorem}

Closely related to the slice Euler characteristic is the \emph{slice genus}
of $L$, defined as $g_s(L)=\min_{F}\text{genus}(F)$.
Since $2g_s(L)\ge 2-n-\chi_s(L)$ where $n$ is the number
of components of $L$, Theorem \ref{main2} provides also a lower bound for $g_s$
(note that this inequality between $g_s$ and $\chi_s$ is not always sharp;
see the end of \S\ref{sect.arf}).
In the case of a knot $K$,
Theorem~\ref{main2} reduces to the estimation
$2g_s(K)\le\beta(K)$ proved by Taylor \cite{Tay}.

Note that the slice Euler characteristic rather than the slice genus
is needed for applications to the study of the topology of plane real algebraic curves
(see \cite{Ore1999, Ore2002}).

The rest of the subsection is devoted to a proof of Theorem~\ref{main2}.
We need the following lemma.

\begin{lemma} \label{cobound2}
If $L$ bounds a locally flat surface $F$ in $B^4$ and a Seifert surface $S$ in $S^3$,
then there is a compact oriented $3$-manifold $M$ in $B^4$ with $\partial M= F \cup_L S$. 
\end{lemma}

The proof of Lemma \ref{cobound2} is folklore; the detailed proof (for knots) can be found
in \cite[Section 1.5]{FL}. However the same arguments work for links and follow essentially
from topological transversality \cite[Chapter 9]{FQ}.

\begin{proof}[Proof of Theorem \ref{main2}]
By Lemma \ref{cobound2}, the closed surface  $F \cup_L S$ bounds a $3$-manifold $M$
in $B^4$. Let $i: F \cup_L S \hookrightarrow M$ and $j:S \hookrightarrow F \cup_L S$
be the inclusions. Consider the composition
$$
   H_1(S) \stackrel{j_*} \longrightarrow H_1(F \cup S) \stackrel{i_*}
                         \longrightarrow H_1(M).
$$
The Seifert form $\theta$ on $H_1(S)$ vanishes on $\ker (i_* \circ j_*)$.
Indeed, suppose that $x$ and $y$ are two elements of $\ker (i_* \circ j_*)$.
There exist 2-chains $\alpha$ and $\beta$ in $M$ with $\partial \alpha=x$ and
$\partial \beta=y$. Then, as in the proof of Theorem~\ref{main1}, we push $\beta$
off from $M$ and we obtain $\theta(x,y)=0$.
Thus $\ind A\ge \dim\ker(i_*\circ j_*)$ whence, by Lemma~\ref{linalg},
$$
   \ind A \ge \dim H_1(S) - \dim\im i_*.
$$
By Lemma \ref{lagr} we also have
$2\dim\im i_*=\dim H_1(F\cup S)=2-\chi(F\cup S)$, thus
\begin{align*}
  \beta(A) &= \dim H_1(S) - 2\ind A\\
         &\le 2\dim\im i_* - \dim H_1(S)
           =  \big(2-\chi(F\cup S)\big) - \big(1-\chi(S)\big) = 1-\chi(F).
\end{align*}
This implies that $ \chi(F)  \leq  1 - \beta(L)$.
\end{proof}


\subsection{Alexander polynomial and signatures}\label{alex.sign.null}

Let $L$ be an oriented link. 
The next result was due to Fox and Milnor in the case of knots, and extended to (colored)
links by the first author. Note that since we work with rational Seifert forms,
the statement is slightly weaker than the original ones.

\begin{corollary} \label{FM}  \cite{FM,Flo}
If $L$ bounds a locally flat surface in $B^4$ with Euler characteristic $1$,
then there exists $F \in \mathbb{Z}[t]$ such that $\Delta_L(t)=F(t) \bar F(t)$
(here $\Delta_L$ and $\bar F$ are as defined in Sections~\ref{sect.alex} and
\ref{sect.propor} respectively). 
\end{corollary}

\begin{proof}
Combine Theorem~\ref{main2} and Proposition~\ref{meta}.
\end{proof}

Let $A$ be a Seifert matrix of $L$. Given $\omega \in S^1  \subset \C$,
the \emph{Levine-Tristram signature and nullity} of $L$ are defined as
the signature and nullity of the hermitian matrix
$(1-\omega)A + (1-\bar \omega)\transpose{A}$.
They give rise to functions $\sigma_L,\eta_L : S^1 \rightarrow \Z$.

\begin{corollary} \label{MT}  \cite{Mu,Tr}
If $L$ bounds a locally flat surface $F$ in $B^4$ without closed components,
then for all $\omega \in S^1 \setminus 1$, the following inequality holds
$$
   |\sigma_L(\omega) | - \eta_L(\omega) \le 1 - \chi(F).
$$
\end{corollary}

Note that Corollary \ref{MT} was first proven for smooth surfaces.
A proof for locally flat surfaces can be found in \cite{Pow}.

\begin{proof}
Let $A$ be a Seifert matrix of $L$. 
Let $i_-$ and $i_+$ be the number of positive and negative eigenvalues of
$A(\omega)=(1-\omega)A+(1-\bar \omega) \transpose{A}$. Without loss of generality,
suppose that $i_- \leq i_+$. By definition, $\sigma_L(\omega)=i_+-i_-$ and
$i_-+i_+ +\eta(\omega)= \dim V$, and $\beta(L)=\dim V - 2 \ind(V,\theta)$.
Since $\ind(V,\theta) \leq i_- + \eta(\omega)$, one gets
$$
     \beta(L)=i_++i_-+\eta(\omega)-2 \ind(V,\theta)
     \geq i_++i_-+\eta(\omega) - 2i_- -2 \eta(\omega)
     =  \vert \sigma_L(\omega) \vert - \eta(\omega).
$$
By Theorem \ref{main2},
$|\sigma_L(\omega)| - \eta(\omega) \le \beta(L) \le 1 - \chi(F)$.
\end{proof}


\section{Some examples}\label{sect.examples}

\subsection{Levels of lower bounds for $\chi_s$.}

Let $L$ be a link with $\Delta_L \neq 0$. Note that $L$ can be a knot.
Let $A$ be a Seifert matrix of $L$. The coindex of $L$ is denoted $\beta(L)=\beta(A)$,
see Definition \ref{sequiv}. By Theorem \ref{main2}, we have
$$
        \chi_s(L) \geq 1-\beta(L).
$$
It is natural to distinguish the following levels of lower bounds of $\beta(A)$
(on each level we discuss only the bounds which are stronger that those on lower levels).

\begin{itemize}
\item  (Level 1)
The Murasugi-Tristram inequality and the Fox-Milnor Theorem, see Corollaries \ref{FM} and \ref{MT}.
\item  (Level 2)
$\beta(A) \geq \beta(A+\transpose{A})$
(this can be used to prove that $\beta(A)\geq c$ for some constant $c$,
only in the case $c \leq 4$; see \ref{sect.Serre}).
\end{itemize}

Let $\alpha_L$ be the concordance class of the isometric structure of $A$.
The next two levels are applicable only for proving the estimate $\beta(A) \geq c$
for $c=2$ or $c=3$.

\begin{itemize}
\item (Level 3)
Application of Levine's theorem \cite{Lev1} (for $c=2$), or Theorem \ref{main2}
(for $c=3$) using the splitting of $\alpha_L$ according to the factorization
of $\Delta_L$ over $\Bbb Q$.
\item (Level 4)
The same as Level 3 but for a splitting of $\alpha_L$ according to the factorization
of $\Delta_L$ over $\Bbb Q_p$ for some prime $p$.
\end{itemize}


\begin{example}
\emph{A computation of $\chi_s(L)$ using a Level 2 lower bound for $\beta(A)$.}
Let $L$ be the link $L^{8a}_{15}(0,0)$ according to the notation
of LinkInfo \cite{linkinfo}. See Figure~\ref{l2}.

\begin{figure}
	\centering
		\includegraphics[scale=1]{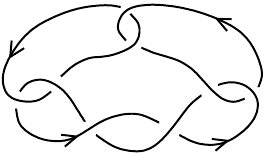} 
	\caption{The link $L^{8a}_{15}(0,0)$.}
	\label{l2}
\end{figure}

The Seifert surface constructed by the standard algorithm via the Seifert circles
of the link diagram in Figure \ref{l2} has Euler characteristic $-3$, thus
$\chi_s(L)\geq -3$. The corresponding Seifert matrix (in some basis) is
$$
  A =  \begin{pmatrix} 
    1&0&0&1\\
    0&1&0&1\\
    0&0&1&1\\
    0&0&0&-1
   \end{pmatrix}.
$$
The Alexander polynomial is $-(t-1)^2(t^2-5t+1)$. All its roots are real and
the signature of $A+\transpose{A}$ is $2$. Hence there is no Level 1 proof that
$\chi_s(L)\le-3$.

Let us show that $\beta(A+\transpose{A})=4$.
The diagonalization of $A+\transpose{A}$ is $[2,2,2,-7/2]$.
The Hilbert symbol over $\Q_2$ of any two
diagonal entries is $1$, hence the corresponding Hasse invariant $\varepsilon$
and the discriminant $d$ (see \S\ref{sect.Serre})
are $1$ and $-28$ respectively. We also have $(-1,-1)=-1\ne\varepsilon$, and
$-28$ is a square in $\Bbb Q_2^\times$ (see \cite[Ch.~II, Thm.~4]{Ser}).
Hence the form is anisotropic (see \cite[Ch.~IV, Thm.~6]{Ser}).
We obtain $\beta(A+\transpose{A})=4$ and then $\chi_s(L)\leq -3$.
This implies $\chi_s(L)=-3$. We also deduce that $g_s(L)=1$.
\end{example}

\medskip


\begin{example}\label{ex.Liv}
\emph{ A bound $\beta(L)>0$ of Level 3 for a knot $L$.}
The idea of this example was suggested by Charles Livingston.
Let $K=12_{803}^a$ be the rational knot $21/2$ in Figure~\ref{l3}.
Let $J$ be the $(3,1)$-cable along $K$.

\begin{figure}
	\centering
		\includegraphics[scale=1]{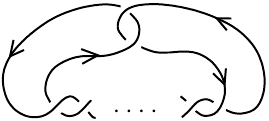} 
	\caption{The $21/2$ rational knot $12^a_{803}$.}
	\label{l3}
\end{figure}

The standard construction via
Seifert circles according to the knot diagram yields a
genus $1$ Seifert surface $S_K$ of $K$. By taking three parallel copies of it and
connecting them by two half-twisted bands (see Figure \ref{surf}), we obtain a
Seifert surface $S_J$ of $J$.

\begin{figure}
	\centering
		\includegraphics[scale=1]{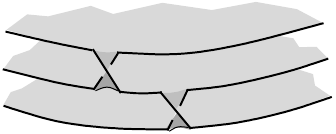} 
	\caption{Seifert surface of the $(3,1)$-cable.}
	\label{surf}
\end{figure}

The corresponding Seifert forms are
$$
    \theta_K=\theta=\begin{pmatrix} 1 & 1 \\ 0 & -5 \end{pmatrix},
    \qquad
    \theta_J= \begin{pmatrix}
                        \theta  &            \theta  & \theta\\
             \transpose{\theta} &            \theta  & \theta\\
             \transpose{\theta} & \transpose{\theta} & \theta\end{pmatrix}.
$$
(for the base of $H_1(S_J)$ we choose three parallel copies of the base of $H_1(S_K)$).
The Alexander polynomials are
$$
    \Delta_K=5t^2-11t+5, \qquad \Delta_J=5t^6-11t^3+5.
$$
As shown in \cite[Corollary~23(c)]{Lev1}, $\theta_K$ represents an element of order $4$
in $\mathcal{G}_\Q$.
Let $L=K\# K\# J\# J$.
The symmetrized forms $\theta_K+\transpose{\theta}_K$ and
$\theta_J+\transpose{\theta}_J$ are
concordant over $\Bbb Q$. Indeed,
$$
   \transpose{P}(\theta_J+\transpose{\theta}_J)P
   =(\theta_K+\transpose{\theta}_K)\oplus \begin{pmatrix}
                     0 & \theta-\transpose{\theta} \\
                     \transpose{\theta}-\theta & 0\end{pmatrix},
  \qquad
  P = \begin{pmatrix}\;\,I & -I & -I \\ -I &\;\,I & 0 \\ \;\,I & 0 &\;\,I \end{pmatrix}.
$$
Hence the symmetrized Seifert form of $L$ is metabolic, i.e.,
there is no Level 2 proof that $L$ is not slice.
However, the $(\Delta_K)$-primary component of $\theta_L$ is non-trivial
in $\mathcal{G}_\Q$ because the order of $\theta_K$ in $\mathcal{G}_\Q$ is $4$.
Hence $L$ is not algebraically slice by \cite{Lev1}.
\end{example}
\medskip


\medskip
\begin{example}
\emph{A bound $\beta(L)>1$ of Level 3}.
Let $L_1=L\,\#\,l$ where $L$ is as in Example~\ref{ex.Liv} and $l$ is
any link such that $\chi_s(l)=0$ and $\Delta_l$ is coprime with $\Delta_L$ 
(for example, the two-component Hopf link).
Then $\chi_s(L_1)<0$ by Theorem~\ref{coindex1} because the
the $(\Delta_K)$- and $(\Delta_J)$-primary components
are non-trivial.
\end{example}


\medskip
\begin{example}\label{ex.(t-1)}
\emph{A Level 3 bound $\beta(L)>1$ based on the $(t-1)$-primary component}.
Let $L=K\,\#\,l_{-1}\,\#\,l_3\,\#\,l_5$ where $K$ is the figure-eight knot and
$l_k$ is the $(2,2k)$-torus link oriented so that its Seifert matrix is the
$(1\times 1)$-matrix $(k)$ (it does not matter which pairs of components are
joint by the connected sums). A Seifert matrix of $L$ is
$A=\diag(-1,3,5)\oplus\left(\smallmatrix 1 & 1\\0 &-1\endsmallmatrix\right)$.
All the Levine-Tristram signatures are $\pm1$ and the form $A+\transpose{A}$
vanishes on the subspace spanned by
$(1,0,0,1,0)\transpose{}$ and $(1,0,1,0,2)\transpose{}$. So, there is no
Level 2 proof that $\chi_s(L)<0$. However, the $(t-1)$-primary component of
$A$ is $\diag(-1,3,5)$, and this is an anisotropic form.
Thus $\beta(L)>1$ by Theorem~\ref{coindex1}, and hence $\chi_s(L)<0$.

\medskip
Note that the bound $\chi_s(L)<0$ in Example~\ref{ex.(t-1)} can be obtained in
a much easier way. Indeed, $L$ has four components, hence any surface $F$ in
$B^4$ with $\partial F=L$ and $\chi(F)=0$ has at least two components, whereas $L$
is not a disjoint union of two sublinks with zero linking number between them.

In general, for links with more than two components, Theorem~\ref{coindex1} hardly
can give bounds better than those coming from the consideration of the sublinks
bounding components of $F$, in particular, using the techniques from
\cite{Flo}, \cite{CF}.
\end{example}


\medskip
\begin{example}\label{ex.lvl4}
              \emph{A bound $\beta(L)>1$ of Level 4}.
Let $L$ be the Montesinos two-component link
$\big(\frac{1}{18}; \frac{11}{4}; -\frac{1}{2}\big)$ oriented as in Figure~\ref{l4}.
On the right hand side in Figure~\ref{l4} we show a ribbon surface for $L$
with Euler characteristic $-2$, thus $\chi_s(L)\ge-2$. We are going to show
that $\chi_s(L)<0$ and thus $\chi_s(L)=-2$ (note that the bound $\chi_s(L)<0$ also
follows from Theorem~\ref{th.Rob} below because $\Arf(L)=1$).

Consider the Seifert surface $S$ shown in Figure~\ref{l4}. Choose the basis of
$H_1(S)$ formed by the loops projected to the boundaries of the regions marked
by $1,\dots,5$ in Figure~\ref{l4}. The Seifert matrix $A$ and the corresponding
monodromy $T=A^{-1}\transpose{A}$ in this basis are
$$
  A = \begin{pmatrix}
  8& 1&  0&  0&  0 \\
  1& 0&  0&  0&  0 \\
  0& 1& -1&  1&  0 \\
  0& 0&  0& -1&  1 \\
  0& 0&  0&  0& -1 \end{pmatrix},
\qquad
  T = \begin{pmatrix}
  1& 0&  1&  0& 0 \\
  0& 1& -8&  0& 0 \\
  0& 1& -8&  0& 1 \\
  0& 0& -1&  0& 1 \\
  0& 0&  0& -1& 1 \end{pmatrix}.
$$
The form $Q=A+\transpose{A}$ vanishes on the subspace spanned by
$(0,1,1,0,0)\transpose{}$ and
$(1,0,2,0,2)\transpose{}$.
Hence $\beta(Q)=1$ and there is no Level 2 proof that $\chi_s(L)<0$.

\begin{figure}
	\centering
		\includegraphics[scale=0.6]{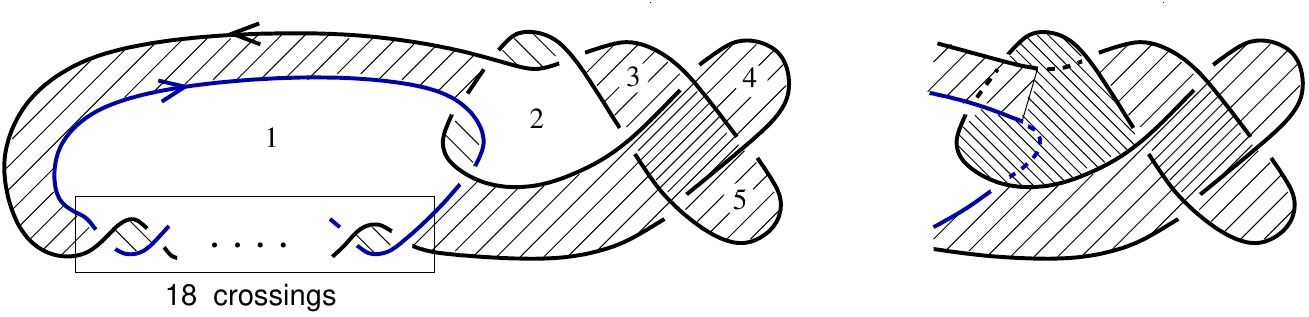} 
	\caption{Montesinos link $\big(\frac{1}{18}; \frac{11}{4}; -\frac{1}{2}\big)$; 
                 its Seifert and ribbon surfaces.}
	\label{l4}
\end{figure}

We have the following factorization over $\Q$:
$$
   \Delta_A = \Delta_T = 1+5t-12t^2+12t^3-5t^4-t^5 = (1-t)(1+6t-6t^2+6t^3+t^4).
$$
Denote the irreducible factor of degree $4$ by $\delta$.
Let $A_0$ be the standard coindex-one matrix for $\Delta_A$, and let
$Q_0=A_0+\transpose{A}_0$ and
$T_0=A_0^{-1}\transpose{A}_0$ (the companion matrix of $\Delta_A$).

The $(T-I)$-invariant vector is $(1,0,0,0,0)\transpose{}$,
hence the $\delta(T)$-invariant subspace is generated by the last four columns
of the matrix $T-I$.
Similarly we find the primary components for $T_0$
 (see~\S\ref{sect.t-1} for the $(t-1)$-component).
The $\delta$-components of $Q$ and $Q_0$ are:
$$
  Q_\delta = \begin{pmatrix}
  -2& 9& -1& -1 \\ 9& -18& 8& 2 \\ -1& 8& -2& 0 \\ -1& 2& 0& -2 \end{pmatrix},
  \qquad
  Q_{0,\delta} = \begin{pmatrix}
  0& -1& 6& -41 \\ -1& 0& -1& 6 \\ 6& -1& 0& -1 \\ -41& 6& -1& 0 \end{pmatrix}.
$$
We have $\det Q_\delta = \det Q_{0,\delta} = -128$ and the Hasse invariant of each
matrix is $-1$ in $\Q_2$ and $1$ in $\Q_p$ for $p\ne 2$. Thus there is no Level~3
bound for $\chi_s(L)$.

The factorization of $\delta$ over $\Q_2$ is $\delta=\delta_1\delta_2$ where
$$
  \delta_1=t^2 - 20t + 1\; (\text{\rm mod}\; 2^8), \qquad
  \delta_2=t^2 + 26t + 1\; (\text{\rm mod}\; 2^8).
$$
The $\delta_1$-primary components of $Q$ and $Q_0$ are
congruent to the diagonal matrices
$$
   Q_{\delta_1}\cong\diag(26,18)\; (\text{\rm mod}\; 32), \quad
   Q_{0,\delta_1}\cong\diag(36,20)\; (\text{\rm mod}\; 64).
$$
Their Hasse invariants are respectively $-1$ and $1$, whence $\beta(A)=1$ and
$\chi_s(L)=-2$.
The following is the {\tt pari-gp} code used for these computations:
\smallskip

\vbox{\begin{verbatim}
  A0=[0,0,1,0,0; 0,0,5,1,0; 0,0,-12,5,1; 1,5,0,0,0; 0,1,0,0,0]^-1;
  A=[8,1,0,0,0; 1,0,0,0,0; 0,1,-1,1,0; 0,0,0,-1,1; 0,0,0,0,-1];
  Q=A+A~; T=1/A*A~; Q0=A0+A0~; T0=1/A0*A0~; /* define A,A0,T,T0,Q,Q0 */ 
  de = t^4 + 6*t^3 - 6*t^2 + 6*t + 1 ;      /* define delta(t)       */
  fp=factorpadic(de,2,8);  /* factorize delta in Q_2[t] up to O(2^8) */
  de1=fp[1,1]; de2=fp[2,1];                 /* define delta1, delta2 */
  V=matker(subst(de1,t,T));           /* compute a 5x2-matrix whose  */
                                      /* columns span ker(delta1(T)) */
  Q1=V~*Q*V;       /* compute the restriction of Q to ker(delta1(T)) */
  D=qfgaussred(Q1);                                /* diagonalize it */
  hilbert(D[1,1],D[2,2],2)            /* compute the Hasse invariant */
  V=matker(subst(de1,t,T0)); /* do all these for Q0, ker(delta1(T0)) */
  D=qfgaussred(V~*Q0*V); hilbert(D[1,1],D[2,2],2)
\end{verbatim}
} 

\end{example}


\begin{example}  \emph{A Level 4 lower bound $\beta(L)>1$ for links
                         with zero Arf invariant.} \label{ex.ar}
Let
$$
   A_1 = \begin{pmatrix}
              -2 &\;\;0 &\;\;0 &\;\;0 &\;\;0 \\
           \;\;0 &\;\;0 &\;\;0 &\;\;1 &\;\;0 \\
           \;\;0 &\;\;1 &\;\;0 &\;\;1 &\;\;1 \\
           \;\;0 &\;\;1 &\;\;1 &\; 10 &   -1 \\
              -2 &\;\;0 &\;\;0 &\;\;0 &   -1 \end{pmatrix},
\qquad
   A_2 = \begin{pmatrix}
           \;\;0 &\;\;0 &   -3 &\;\;1 &   -1 \\
              -3 &\;\;0 &\;\;0 &   -3 &\;\;1 \\
           \;\;1 &   -3 &\;\;0 &\;\;0 &   -3 \\
              -1 &\;\;1 &   -3 &\;\;0 &\;\;0 \\
              -1 &   -1 &\;\;1 &   -3 &\;\;0 \end{pmatrix}.
$$
One easily checks that
$\transpose{U}_j\!(A_j-\transpose{A_j})U_j=(0)\oplus
   \left(\smallmatrix 0 & 1 \\ -1 & 0 \endsmallmatrix\right) \oplus
   \left(\smallmatrix 0 & 1 \\ -1 & 0 \endsmallmatrix\right)$
for some integral matrices $U_1$, $U_2$.
Hence, by \cite[Theorem~5.1.4]{KawBook},
$A_j$ is a Seifert matrix of some 2-component link $L_j$.
These links are rather big, for example, the linking number between
the components is $38$ for $A_1$ and $-110$ for $A_2$. We have
$$
  \Delta_{L_1}=2(1+7t-7t^2+7t^3-7t^4-t^5),\qquad
  \Delta_{L_2}=10(1-5t-t^2+t^3+5t^4-t^5).
$$
The pairs of vectors $(e_1,-2e_3+e_4+2e_5)$ and $(e_2, e_5)$
generate totally isotropic
subspaces for $A_1+\transpose{A}_1$ and $A_2+\transpose{A}_2$,
thus there is no Level 2 bound for $\chi_s(L_j)$.
There is no Level 3 bound neither. However,
there are Level 4 bounds $\chi_s(L_j)<0$
coming from the factorizations of $\Delta_{L_j}$ over $\Q_2$ for $L_1$ and
over $\Q_5$ for $L_2$. The computations are the same as in
Example~\ref{ex.lvl4}.

In contrary to Example~\ref{ex.lvl4}, the Arf-Robertello invariant
(see~\S\ref{sect.arf}) vanishes for $L_1$ and $L_2$, and we do not know
any proof of the bound $\chi_s(L_j)<0$ without Theorem~\ref{coindex1}.
\end{example}


\subsection{Arf-Robertello invariant: $g_s(L)$ vs.~$\chi_s(L)$}\label{sect.arf}

The {\it Arf invariant} $\Arf(K)$ of a knot $K$ can be defined as the
Arf invariant (see, e.g.~\cite{MH}) of the quadratic
form $q:H_1(S;\F_2)\to\F_2$, $q(x)=\theta(x,x)$, where $\theta$ is
the reduction mod 2 of a Seifert form on a Seifert surface $S$ of $K$
(see \cite{Rob}). By \cite{Mu2}, $\Arf(K)=0$ iff $\Delta_K(-1)\equiv\pm1\mod 8$.

Following Robertello \cite{Rob}, we say that a knot $K$ is {\it related} to
an oriented link $L$ if there exists a locally flat surface $F$ of genus 0
in $S^3\times I$ such that $\partial F \cap (S^3\times 0)=L$
and $\partial F\cap(S^3\times 1)=K$.
A link $L=k_1\cup\dots\cup k_\mu$ ($k_i$ are knots) is called in \cite{Rob}
{\it proper} if all the linking numbers $\lk(k_i,L\setminus k_i)$ are even.

Let $L$ be a proper link. By \cite[Theorem 2]{Rob}, the Arf invariant
of any knot related to a proper link $L$ is an invariant of $L$. 
We call it the {\it Arf-Robertello invariant} of $L$ and denote by $\Arf(L)$. 
Note that some interesting interpretations of the
Arf-Robertello invariant are found in \cite{Mou} for algebraically split links.
If $g_s(L)=0$, then the unknot is related to $L$, which gives the following
lower bound for the slice genus.

\begin{theorem} \label{th.Rob}
{\rm(Essentially \cite{Rob})}
If $L$ is a proper link and $g_s(L)=0$, then $\Arf(L)=0$.
\end{theorem}

Given a diagram of a proper link $L$ with $\mu$ components,
one can compute $\Arf(L)$ using any knot obtained from $L$ by resolving
$\mu-1$ crossings. $\Arf(L)$ can be also computed from a Seifert form of $L$.

\begin{proposition} \label{prop.arf}
Let $S$ be a Seifert surface of a proper link $L$,
and $\theta$ be the Seifert form on $V=H_1(S;\F_2)$. Let $q(x)=\theta(x,x)$
be the corresponding quadratic form on $V$.
Let $W$ be any complementary subspace for $\ker\theta$, i.e.,
$V=W\oplus\ker\theta$. Then $\Arf(L)$ is the Arf invariant of $q|_W$.
\end{proposition}

As we pointed out in \S\ref{bound.chi_s}, for an $n$-component link $L$,
the bound $2g_s(L)\ge 2-n-\chi_s(L)$ is not always sharp.
A strict inequality takes place when the maximal Euler characteristic
cannot be attained on a connected surface in $B^4$.

The simplest example is the Whitehead link $L=L^{5a}_1$
(the braid closure of $\sigma_1\sigma_2^{-1}\sigma_1\sigma_2^{-1}\sigma_1$).
By switching the central crossing $\sigma_1$,
one obtains a trivial 2-component link.
Hence $L$ bounds a disjoint union of a disk and a genus 1 surface.
Thus $\chi_s(L)=0$ and $g_s(L)\le 1$. On the other hand,
resolving one of the crossings $\sigma_2^{-1}$, one obtains a trefoil.
Hence $\Arf(L)=1$ and $g_s(L)\ge 1$. So, we obtain $2g_s(L)=2>0=2-n-\chi_s(L)$.


\subsection{What is a good definition of algebraically slice Euler characteristic}
\label{sect.slice.chi}

As shown in \cite[Theorem~3]{Tay}, if $A$ is a Seifert matrix of a knot $K$,
then there exists a knot $K'$ whose ribbon genus (and hence the slice genus as well)
is $\frac12\beta(A)$. So, it is natural to call $\frac12\beta(A)$ the
{\it algebraic slice genus} of $A$ and of $K$.

The corresponding statement for links is false as one can see in the following example.

\begin{example}\label{ex.arf}
Let $L$ be the connected sum of the trefoil and
the torus link $T(2,-8)$ oriented so that the Seifert matrix of $L$ is
$A=(-4)\oplus\left(\smallmatrix 1&1\\0&1\endsmallmatrix\right)$.
We have $\transpose{x}\!Ax=0$ for $x=(1,2,0)\transpose{}$, hence $\beta(A)=1$.
Let us show that there does not exist any link $L'$ such that
$A$ is a Seifert matrix of $L'$ and $\chi_s(L')=0$.
Indeed, suppose that such a link $L'$ does exist. The kernel of $A-\transpose{A}$ is
generated by $x_0=(1,0,0\transpose{)}$, hence $L'=k_1\cup k_2$ is a $2$-component link,
$\pm x_0$ is the homology class of $k_1$, and $\lk(k_1,k_2)=-4$, in particular,
$L'$ is a proper link (see~\S\ref{sect.arf}). By Proposition~\ref{prop.arf} we have
$\Arf(L')=\Arf(K)=1$, hence $g_s(L')>0$ by Theorem~\ref{th.Rob}.
Since $\lk(k_1,k_2)\ne 0$, any surface in $B^4$ bounding $L'$ is connected, whence
$\chi_s(L')<0$.
\end{example}

A candidate for a good notion of algebraic slice Euler characteristic of
a link with Seifert matrix $A$ is $1-\bsat(A)$ where
$\bsat(A)=n-2k$, $n$ is the size of $A$, and $k$ is the maximal rank of a totally
isotropic subgroup $U$ of $\Z^n$ such that $U+\ker(A-\transpose{A})$ is
a saturated subgroup of $\Z^n$ (a subgroup $L$ of $\Z^n$ is called
{\it saturated} if $\Z^n/L$ is torsion-free).

We do not know whether $\bsat$ is invariant even under S-equivalence.
We neither know any upper bound for $\bsat(A)$ better than $\beta(A)$. 
However, the following two observations (Propositions~\ref{prop.rib1} and
\ref{prop.rib2}) give a hope that $\bsat$ could play some r\^ole.
We define the {\it ribbon Euler characteristic} $\chi_r(L)$ of a
link $L$ as the maximum of $\chi(F)$ over all ribbon immersions of surfaces
without closed components $r:F\to S^3$ such that $L=r(\partial F)$.
It is well-known that $\chi_s(L)\ge\chi_r(L)$

\begin{proposition}\label{prop.rib1}
Any link $L$ has a Seifert matrix $A$ such that $\chi_r(L)\le 1-\bsat(A)$.
\end{proposition}

\begin{figure}
	\centering
		\includegraphics[scale=0.42]{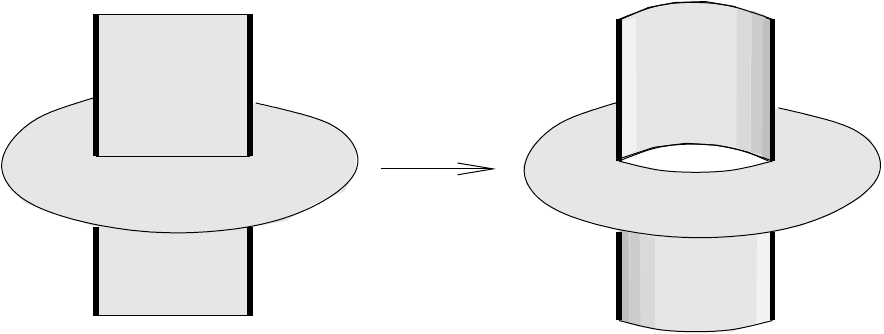} 
	\caption{Kauffman's trick; cf.~\cite[pp.~212--213]{Kau}.}
	\label{kauf}
\end{figure}

\begin{proof}
Let $r:F\to S^3$, $r(\partial F)=L$, be a ribbon immersion with $\chi(F)=\chi_r(L)$.
Without loss of generality we may assume that $r(F)$ is connected.
Using Kauffman's trick (see Figure~\ref{kauf}), we transform $F$ into
a Seifert surface $S$. Denote the Seifert form on $H_1(S)$ by $\theta$.
Let $x_1,\dots,x_k$ be the simple closed loops corresponding to
the boundaries the disks depicted in Figure~\ref{kauf}. They generate a totally
isotropic subgroup $U$ of $H_1(S)$ of rank $k$. It is clear that $\chi(S)=\chi(F)-2k$.
The kernel of $\theta-\transpose{\theta}$ is generated by the components of $\partial L$.
Thus $U+\ker(\theta-\transpose{\theta})$ is generated by pairwise disjoint simple
closed curves on $S$, hence it is a saturated subgroup of $H_1(S)$ whence the result.
\end{proof}

\begin{proposition} \label{prop.rib2}
Let $A$ be a Seifert matrix of a link $L$. Then there exists a link $L'$
such that $\chi_r(L')\ge 1-\bsat(A)$ and $A$ is a Seifert matrix of $L'$.
\end{proposition}

\begin{proof} The construction of a ribbon surface is more or less the same
as in the proof of \cite[Theorem~3]{Tay}. Let $(V,\theta)$, $V=\Z^n$,
be the bilinear form defined by $A$ and let $\iota=\theta-\transpose{\theta}$
(the intersection form on the Seifert surface of $L$). Let $K=\ker\iota$ and
let $U$ be a totally isotropic subgroup of $V$ of the appropriate rank such that 
$U+K$ is saturated. Since $K$ is generated by the boundary components of a Seifert
surface of $L$, it is saturated. It is clear that $U$ is saturated as well.
Then we have
$V=V_1\oplus\dots\oplus V_4$ where $V_1\oplus V_2=K$ and $V_2\oplus V_3=U$.
Since the restriction of $\iota$ to $V_3\oplus V_4$ is
unimodular, $V_4$ splits into $V_4'\oplus V_4''\oplus V_4'''$ so that, for some bases
$(x_1,\dots,x_p)$, $(x'_1,\dots,x'_p)$, $(x''_1,\dots,x''_q)$, $(x'''_1,\dots,x'''_q)$
of, respectively, $V_3$, $V'_4$, $V''_4$, $V'''_4$, we have
$\iota(x_i,x'_j)=\delta_{ij}$, $\iota(x''_i,x'''_j)=\delta_{ij}$, and $V_3\oplus V'_4$
is $\iota$-orthogonal to $V''_4\oplus V'''_4$.
Let $(u_1,\dots,u_r)$ and $(v_1,\dots,v_s)$ be some bases of $V_1$ and $V_2$
respectively.

\begin{figure}
	\centering
		\includegraphics[scale=0.6]{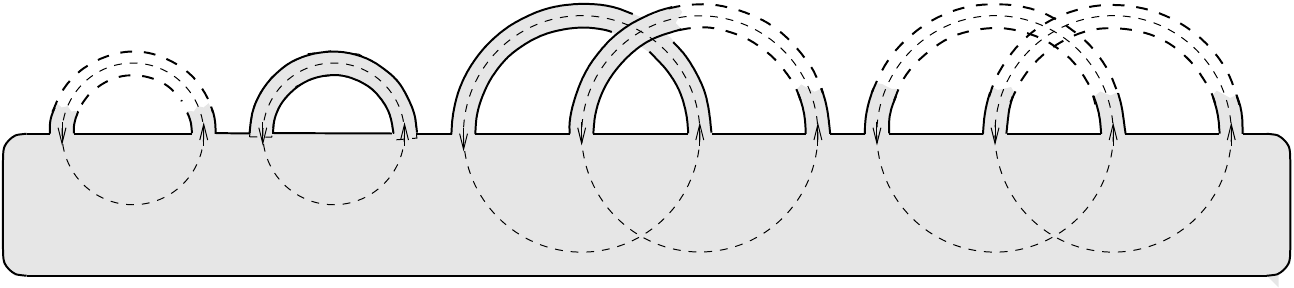} 
        \put(-360,25){$u_i$}\put(-303,25){$v_i$}\put(-247,23){$x_i$}
        \put(-214,23){$x'_i$}\put(-130,23){$x''_i$}\put(-98,23){$x'''_i$}
	\caption{Seifert surface in the proof of Proposition \ref{prop.rib2}.}
	\label{fig.seif}
\end{figure}

\begin{figure}
	\centering
		\includegraphics[scale=0.6]{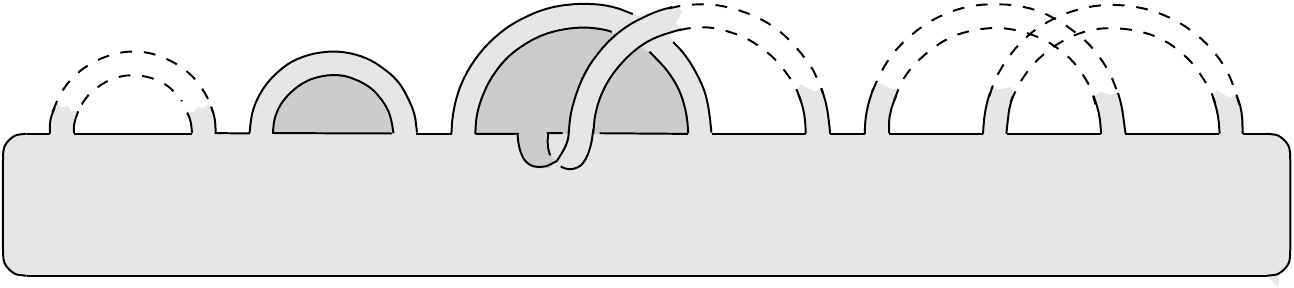} 
	\caption{Ribbon surface in the proof of Proposition \ref{prop.rib2}.}
	\label{fig.ribb}
\end{figure}

By the classical Seifert's method \cite{Sei}
(see also \cite[Theorem~5.1.4]{KawBook}), we construct $L'$ as the
boundary of an embedded surface $S$ shown in Figure~\ref{fig.seif} where
the dashed parts of the bands are supposed to be twisted and linked with each other and
with the solid bands so that they realize the given Seifert form.
Then, near the bands corresponding to the cycles $v_i$ and $x_i$, we modify $S$
as shown in Figure~\ref{fig.ribb}.
After this modification, each band that was linked with $v_i$ and/or $x_i$ gets
ribbon intersections with the modified piece of the surface.
\end{proof}

\subsection{On Milnor's remark}

The results of Levine \cite{Lev1} are based, in particular, on
Milnor's theorem \cite[Thm.~2.1]{Mil} which states that an
isometric structure $(V,b,T)$ over a field $\F$ is determined up to isomorphism
by $(V,b)$ and the minimal polynomial for $T$ under the condition that
$\F$ is $\R$ or $\Q_p$, and the minimal polynomial is irreducible.

It is said in \cite[p.~87]{Mil}:
{\sl ``Note that the corresponding statement for an arbitrary field would
definitely be false (e.g.~for the field of rational numbers)''.}
The following example (proposed by Andrey Levin) illustrates that
this statement is indeed false for $\F=\Q$.

\begin{example}
Let $\rho=\exp(2\pi i/5)$, $V=\Q(\rho)$, and let $T:V\to V$ be given by
$x\mapsto\rho x$.
For $k\in\{1,2\}$, we set $a_k=\rho^k+\rho^{-k}=2\cos(2\pi k/5)$
and we define the symmetric bilinear
form $b_k:V\times V\to\Q$ as $b_k(x,y)=\text{trace}(a_k x\bar y)$.
Then $T$ preserves $b_k$, the bilinear forms $b_1$ and $b_2$ are congruent, and 
the minimal polynomial of $T$ is $\Phi_5(t)=t^4+t^3+t^2+t+1$ which is irreducible
over $\Q$.

Let us show that the isometric structures $(V,b_1,T)$ and $(V,b_2,T)$
are not isomorphic.
We have $\Phi_5(t)=\delta_1(t)\delta_2(t)$ where
$\delta_k(t)=t^2-a_k t+1$. Let $V_k=\ker\delta_k(T_{\R})$ where $T_{\R}$
is the $\R$-linear extension of $T$ to $V\otimes\R$.
 Then $V\otimes\R=V_1\oplus V_2$ and
this decomposition is orthogonal for each of the forms $b_1$ and $b_2$
(we keep the notation $b_k$ for the extension of $b_k$
to $V\otimes\R$).
A straightforward computation shows that
the forms $b_1|_{V_1}$ and $b_2|_{V_2}$ are positive definite whereas 
the forms $b_1|_{V_2}$ and $b_2|_{V_1}$ are negative definite.
Hence the discussed isometric structures are not isomorphic.

The matrices of $T$, $b_1$, and $b_2$ in the basis $(1,\rho,\rho^2,\rho^3)$ are,
respectively,
$$
{\def\+{\hskip6pt}
   \begin{pmatrix}
      0 & 0 & 0 & -1 \\
      1 & 0 & 0 & -1 \\
      0 & 1 & 0 & -1 \\
      0 & 0 & 1 & -1 \end{pmatrix},
   \quad
   \begin{pmatrix}
      -2 &\+3 & -2 & -2 \\
     \+3 & -2 &\+3 & -2 \\
      -2 &\+3 & -2 &\+3 \\
      -2 & -2 &\+3 & -2 \end{pmatrix},
   \quad
    \begin{pmatrix}
      -2 & -2 &\+3 &\+3 \\
      -2 & -2 & -2 &\+3 \\
     \+3 & -2 & -2 & -2 \\
     \+3 &\+3 & -2 & -2 \end{pmatrix}.
}
$$
The subspace $V_k$, $k=1,2$, is spanned by the vectors
$(1,1,0,-a_k)$ and $(1,a_k^{-1},1,0)$.
The matrices of $b_1|_{V_k}$ and $b_2|_{V_k}$ in these bases are, respectively,
$$
   5a_k\begin{pmatrix} 2 & 1 \\ 1 & 2a_k^2 \end{pmatrix} \qquad\text{and}\qquad
   -5a_k^{-1}\begin{pmatrix} 2 & 1 \\ 1 & 2a_k^2 \end{pmatrix}.
$$
\end{example}

\bibliographystyle{amsalpha}
\bibliography{concord}
 
\end{document}